\documentclass[a4paper,10pt]{article}

\usepackage{float}
\usepackage{amsfonts,amsmath}
\usepackage{amssymb}
\usepackage[dvips]{graphicx}
\usepackage{amscd}
\usepackage[mathscr]{eucal}

\usepackage[notref]{showkeys}

\usepackage{amsthm}
\newtheorem{theorem}{Theorem}[section]
\newtheorem{definition}[theorem]{Definition}
\newtheorem{proposition}[theorem]{Proposition}
\newtheorem{lemma}[theorem]{Lemma}
\newtheorem{corollary}[theorem]{Corollary}
\newtheorem{remark}[theorem]{Remark}

\def\endproof{\ \hfill\hbox{\vbox{\hrule\hbox{\vrule
height5pt\kern5pt\vrule height5pt}\hrule}}\par\medskip\rm}
\evensidemargin \oddsidemargin 
\newcommand{\be}{\begin{equation}}
\newcommand{\ee}{\end{equation}}

\usepackage{fancyhdr}
\bigskip

\title{{\bf On the dynamics of WKB wave functions whose phase are weak KAM solutions of H-J equation 
}
}

\author{Thierry Paul\footnote{CNRS and CMLS, Ecole Polytechnique 
(Palaiseau), thierry.paul@math.polytechnique.fr} 
\quad Lorenzo Zanelli\footnote{CMLS, Ecole Polytechnique 
(Palaiseau), lorenzo.zanelli@ens.fr}}

\begin{document}

\maketitle

\begin{abstract}
\noindent
In the framework of  toroidal Pseudodifferential operators on the flat torus $\Bbb T^n := (\Bbb R / 2\pi \Bbb Z)^n$ we begin by proving the closure under composition for the class of  Weyl  operators  $\mathrm{Op}^w_\hbar(b)$ with simbols  $b \in S^m (\mathbb{T}^n \times \mathbb{R}^n)$. Subsequently, we consider  $\mathrm{Op}^w_\hbar(H)$ when $H=\frac{1}{2} |\eta|^2 + V(x)$ where $V \in C^\infty (\Bbb T^n;\Bbb R)$  and we exhibit the toroidal version of the equation for the Wigner transform of the solution of the Schr\"odinger equation. Moreover, we prove the convergence (in a weak sense)  of the Wigner transform of the solution of the Schr\"odinger equation to the solution of the Liouville equation on $\Bbb T^n \times \Bbb R^n$ written in the measure sense. These results are applied to the study of some WKB type wave functions in the Sobolev space $H^{1} (\mathbb{T}^n; \Bbb C)$ with phase functions in the class of Lipschitz continuous weak KAM solutions (of positive and negative type) of the Hamilton-Jacobi equation $\frac{1}{2} |P+ \nabla_x v_\pm (P,x)|^2 + V(x) = \bar{H}(P)$ for $P \in \ell \Bbb Z^n$ with $\ell >0$,  and to the study of the backward and forward time propagation of the related Wigner measures supported on the graph of  $P+ \nabla_x v_\pm$.
\end{abstract}

\bigskip

{\bf {\small Keywords}:}   Toroidal Pseudodifferential operators, Wigner measures,  Hamilton-Jacobi equation.

\newpage 

\tableofcontents

%\newpage 

\lhead{}  \chead{}

\lhead{}

\renewcommand{\headrulewidth}{0.5 pt} 

%\newpage
\section{Introduction}  \markboth{Introduction}{}
In this paper we study WKB type wave functions on flat torus $\Bbb T^n := (\Bbb R / 2\pi \Bbb Z)^n$, namely  functions of the form
\be\label{def1}
\psi(x) = a(x) e^{iS(x)/\hbar},\  x\in\mathbb T^n,\  n\geq 1
\ee
where $a = a_{\hbar,P}$ is a family of  functions in $L^2 (\mathbb T^n;\Bbb R)$ and $S(x) = P \cdot x+v(x),\ P \in \ell \Bbb Z^n$, $\ell > 0$,  $\hbar^{-1} \in \ell^{-1} \Bbb N$,  the phase $v(x)=v(P,x)$ is a Lipschitz continuous weak KAM solution of the  stationary Hamilton-Jacobi equation 
\begin{equation}
\label{def-eff0-intro}
H(x,P+\nabla_x v(P,x))  =  \bar{H}(P), 
%\quad P \in \Bbb R^n, 
\end{equation}
for Hamitonian  $H(x,\xi) := \frac{1}{2} |\xi|^2 + V(x)$,   $V \in C^\infty (\mathbb{T}^n)$,  see Section \ref{sub-wHJ} for precise definitions.

It is well known that in the case where $v$ is a regular function,  the wave function $\psi$ is, under general conditions on the family $a=a_{\hbar,P}$, a Lagrangian distribution associated to the Lagrangian manifold $\Lambda_P :=\{(x,\eta) \in \Bbb T^n \times \Bbb R^n,\ \eta = P +\nabla_x v(P,x)\}$. Therefore it has an associated monokinetic  Wigner measure of the form
\be\label{mono1}
dw(x,\eta)=\delta(\eta-(P+\nabla_x v(P,x))) |a_0(x)|^2 dx  .
\ee
Moreover it remains of the same type under propagation through the Schr\"odinger equation whose quantum Hamiltonian is the quantization of the function $H(x,\xi)$ (see Section \ref{Weyl} for details on the toroidal quantization) leadding to a Wigner measure 
\be\label{wig2}
dw^t (x,\eta) =  \delta(\eta-(P+\nabla_x v(P,x))) |a_0^t(x)|^2dx
\ee
where $|a_0(x)|^2$ satisfies a transport equation in such a way that $dw^t$ is the pushforward of $dw$ by the Hamiltonian flow of $H$.

The goal of this paper is to show what remains of this construction in the case where $v$ is a solution of \eqref{def-eff0-intro} with  only a Lipschitz continuity property, a regularity far for being used in the framework of standard microlocal analysis. 

Note that propagation of monokinetic Wigner measures with low regularity momentum profiles and application to the classical limit of propagation of WKB type wave functions have been recently studied in \cite{bgmp}. The regularity assumption in \cite{bgmp} is much stronger than ours, but at the contrary the construction in \cite{bgmp} works for any profile with a given regularity as we need our phase function to be a solution of the Hamilton-Jacobi equation. Therefore the two papers are complementary.

The precise definition of our WKB state, especially of the amplitude in \eqref{def1}, is given in Section \ref{SEC-wkb}, Definition \ref{wave-d} where a family of examples are given in the remark \ref{EX-a} following the definition. 

Note that WKB states on the torus with phase functions issued form weak KAM theory have been used in \cite{E1}, \cite{E2} where it has been studied $L^2$-energy quasimode estimates. In \cite{P-Z} a class of WKB states on the torus with regularized phase function have been defined in such a way the associated Wigner measures are coinciding with the Legendre transform of the so-called Mather measures. 

In the present paper we will work with the true solution of Hamilton-Jacobi equation for the phase and will use a kind of regularization for the amplitude, as no canonical function choice is offered for the latter out 
%of invariant measures provided by 
weak KAM theory.
 
 Our first main result concerns the Wigner measure $w$, as defined in Section \ref{wig1}, Definition \ref{def-wc1}, associated to our family of WKB states. It claims, Theorem \ref{TH4}, that $w$ is as expected monokinetic in the sense that it has the form
\be
\label{bof}
dw(x,\eta) =  \delta(\eta - (P+\nabla_x v(P,x))) dm_P (x)
\ee
 where the limit in the measure sense $dm_P (x) = \lim\limits_{\hbar\to 0} |a_{\hbar,P}(x)|^2dx$ exists by Definition \ref{wave-d}. In fact, we also assume that 
 $dm_P \ll \pi_\star(dw_P) =: d\sigma_P$ where $dw_P$ is the Legendre transform of a Mather $P$-minimal measure (see Section \ref{sec-M}).
 This setting implies that any measure $dw(x,\eta)$ as in (\ref{bof}) is asbolutely continuous to $dw_P$ itself, as shown in Lemma \ref{Lem-ac}. We also underline that $d\sigma_P$ solves the continuity equation
 \begin{equation}
 \label{cont-0}
 0 = \int_{\Bbb T^n } \nabla_x f(x) \cdot (P+ \nabla_x v(P,x)) \, d\sigma_P (x) \quad \forall f \in C^\infty (\Bbb T^n), 
 \end{equation}
and this  can be interpreted as the result of an asymptotic  free current density condition for the wave functions $\psi$ of type (\ref{def1}), as we show in Proposition \ref{prop46}. We recall that in the usual construction of WKB wave functions (working with integrability or almost-integrability assumptions on $H$) the determination of the amplitude function $a(x)$ is related to the solution of the continuity equation (\ref{cont-0}) written in the strong sense for the function $\sigma(x) = a^2 (x)$, namely ${\rm div}_x [ (P+ \nabla_x v(P,x)) \sigma(x) ] = 0$.
 
The above assumption on $dm_P$ together with the monokinetic form of $dw_P$ on the graph  of a weak KAM solution of the Hamilton-Jacobi equation allow to  study very much easily the time propagation of such measures, which remains of monokinetic type. This is in fact our second main result, which deals with the classical limit of the Wigner transform of the evolved WKB state.  It is contained within Theorem \ref{th51} and Proposition \ref{prop53} where the propagation, 
\be
\label{bof2}
d w^t (x,\eta) = \delta(\eta - (P+\nabla_x v(P,x)))  g(t,P,x) dm_P (x) 
\ee
both forward and  backward  (they are different in our situation) in time is exhibited. \\

The paper is organized as follows: Section \ref{preli} is devoted to some preliminaries concerning the Weyl quantization on the torus (\ref{Weyl}) and the weak KAM theory (\ref{kam1}). Section \ref{wtd1} concerns the dynamics of the Wigner transform on the  torus and Section \ref{scwt} the classical limit of the Wigner transform, including the Section \ref{SEC-wkb} where the monokinetic property of the Wigner function of our WKB state is established. Its propagtion is studied in the final Section \ref{pwt}.

%\newpage

\noindent
\section{Preliminaries}\label{preli}  \markboth{Preliminaries}{}

\subsection{The Weyl quantization on the torus}
\label{Weyl}
\subsubsection{Settings}
Let us consider the flat torus $\Bbb T^n := (\Bbb R / 2\pi \Bbb Z)^n$.  
The  class of   symbols $b \in S^m_{\rho, \delta} (\mathbb{T}^n \times \mathbb{R}^n)$, $m \in \mathbb{R}$, $0 \le \delta$, $\rho \le1$, consisting of those functions  
in $C^\infty (\mathbb{T}^n \times \mathbb{R}^n;\Bbb R)$ which are  $2\pi$-periodic in $x$ (that is, in each variable $x_j$, $1\leq j\leq n$) and for which
for all $\alpha, \beta \in \mathbb{Z}_+^n$ there exists $C_{\alpha \beta} >0$ such that $\forall$ $(x,\eta) \in \mathbb{T}^n \times \mathbb{R}^n$
\begin{equation}
\label{symb00}
|  \partial_x^\beta \partial_\eta^\alpha  b (x,\eta) |   \le  C_{\alpha \beta m} \langle \eta \rangle^{m- \rho |\alpha| + \delta |\beta|}
\end{equation}
where $\langle\eta\rangle:=(1+|\eta|^2)^{1/2}$. In particular,  the set $S^m_{1,0} (\mathbb{T}^n \times \mathbb{R}^n)$ is denoted by $S^m (\mathbb{T}^n \times \mathbb{R}^n)$. 
\\
The  toroidal Pseudodifferential Operator associated to  $b \in S^m (\mathbb{T}^n \times \mathbb{R}^n)$ reads 
\begin{equation}
b(X,D) \psi(x):=(2\pi)^{-n}\sum_{\kappa \in\mathbb{Z}^n}\int_{\mathbb{T}^n}e^{i\langle x-y,\kappa \rangle}b(y,\kappa)\psi(y)dy, \quad \psi \in C^\infty (\mathbb{T}^n;\Bbb C), 
\end{equation}
see  \cite{R-T}. In particular, we have a map $b(X,D) : C^\infty (\mathbb{T}^n) \longrightarrow \mathcal{D}^\prime  (\mathbb{T}^n)$.  We recall that  $u \in \mathcal{D}^\prime  (\mathbb{T}^n)$ are the linear maps $u: C^\infty (\mathbb{T}^n) \longrightarrow \Bbb C$ such that $\exists$ $C>0$ and $k \in \Bbb N$, for which $|u(\phi)| \le C \sum_{|\alpha|\le k} \| \partial_x^\alpha \phi  \|_\infty$ $\forall \phi \in C^\infty (\Bbb T^n)$, see for example Definition 2.1.1 of \cite{Ho}. 
%It is easily seen that $b(X,D) \psi \in L^2 (\Bbb T^n)$ $\forall \psi \in C^\infty (\mathbb{T}^n)$.\\ 
Given a symbol $b\in S^m(\mathbb{T}^n\times\mathbb{R}^n)$, the (toroidal) Weyl quantization reads
\begin{equation}
\label{weyl}
\mathrm{Op}^w_\hbar(b)\psi(x) := (2\pi)^{-n}\sum_{\kappa \in\mathbb{Z}^n}\int_{\mathbb{T}^n}e^{i\langle x-y,\kappa \rangle}b(y,\hbar\kappa/2)\psi(2y-x)dy,\,\,\,\, \psi\in C^\infty(\mathbb{T}^n).
\end{equation}
Hence, it  follows 
\begin{equation}
\label{eq-O}
\mathrm{Op}^w_{\hbar} (b)  \psi (x) = (b(X,\frac{\hbar}{2}D) \circ T_x \, \psi )(x)
\end{equation}
where $T_x : C^\infty (\mathbb{T}^n) \rightarrow C^\infty (\mathbb{T}^n)$ defined as $(T_x \psi) (y) := \psi (2y-x)$  is linear, invertible and $L^2$-norm preserving.
Starting from quantization in (\ref{weyl}), we  now introduce the  Wigner transform $W_\hbar \psi$  by 
\begin{equation}
W_\hbar \psi (x,\eta) := (2\pi)^{-n}\int_{\Bbb{T}^n} e^{2 \frac{i}{\hbar} \langle z,\eta\rangle}\psi(x-z) \psi^\star (x+z)dz, \quad \eta \in \frac{\hbar}{2} \Bbb Z^n,
\label{WignerT}
\end{equation}
which is well defined also for $\psi \in L^2 (\Bbb T^n)$. 	
For $b \in S^m (\mathbb{T}^n \times \mathbb{R}^n)$ the Wigner distribution reads
\begin{equation}
\label{med4}
\langle \psi , {\rm Op}^w_\hbar (b) \psi\rangle  =  \sum_{\eta \in \frac{\hbar}{2} \mathbb{Z}^n} \int_{\mathbb{T}^n} b(x,\eta) W_\hbar \psi(x,\eta)dx, \quad \psi \in C^\infty (\Bbb T^n).
\end{equation}
For $b \in S^0 (\mathbb{T}^n \times \mathbb{R}^n)$ and  $\psi \in L^2 (\Bbb T^n)$, the mean value $\langle \psi , {\rm Op}^w_\hbar (b) \psi\rangle_{L^2(\Bbb{T}^n)}$ is well defined  thanks to the $L^2$ - boundedness estimate of ${\rm Op}^w_\hbar (b)$,  see Theorem  \ref{Th-Bound0}.

\begin{remark}
Before to recall the notion of toroidal symbols and toroidal amplitudes, we need  first to remind the notion of partial difference operator $\triangle$. Given $f : \Bbb Z^n_\kappa \longrightarrow \Bbb C$, it is defined the 
\begin{equation}
\triangle_{\kappa_j} f(\kappa) := f(\kappa + e_j) - f(\kappa) 
\end{equation}
where $e_j \in \Bbb N^n$, $(e_j)_j =1$ and  $(e_j)_i =1$ if $i \neq j$. The composition  provide $\triangle_{\kappa}^\alpha f(\kappa) := \triangle_{\kappa_1}^{\alpha_1} f(\kappa)...\triangle_{\kappa_n}^{\alpha_n} f(\kappa)$ for any $\alpha \in \Bbb N_0^n$.
We recall now that  toroidal simbols $\widetilde{b} \in S^m_{\rho, \delta} (\mathbb{T}^n \times \mathbb{Z}^n)$, $m \in \mathbb{R}$, $0 \le \delta$, $\rho \le1$, are those functions  
which are smooth in $x$ for all $\kappa \in \Bbb Z^n$, $2\pi$-periodic in $x$ and for which
for all $\alpha, \beta \in \mathbb{Z}_+^n$ there exists $C_{\alpha \beta m} >0$ such that $\forall$ $(x,\kappa) \in \mathbb{T}^n \times \mathbb{Z}^n$
\begin{equation}
\label{symb-D}
|  \partial_x^\beta \triangle_\kappa^\alpha \  \widetilde{b} (x,\kappa) |   \le  C_{\alpha \beta m} \langle \kappa \rangle^{m- \rho |\alpha| + \delta |\beta|}
\end{equation}
where $\langle \kappa \rangle:=(1+|\kappa|^2)^{1/2}$. As usually,  $S^m (\mathbb{T}^n \times \mathbb{Z}^n)$ stands for $S^m_{1,0} (\mathbb{T}^n \times \mathbb{Z}^n)$. In the same way, it is defined the set of toroidal amplitudes $S^m_{\rho, \delta} (\mathbb{T}^n \times \Bbb T^n \times \mathbb{Z}^n)$.
\\
The link between this class of simbols and the euclidean ones $S^m_{\rho, \delta} (\mathbb{T}^n \times \mathbb{R}^n)$ is shown within Theorem 5.2 in \cite{R-T}. More precisely, for  any $\widetilde{b} \in S^m_{\rho, \delta} (\mathbb{T}^n \times \mathbb{Z}^n)$ there exists $b \in S^m_{\rho, \delta} (\mathbb{T}^n \times \mathbb{R}^n)$   such that $\widetilde{b} = b|_{\mathbb{T}^n \times \mathbb{Z}^n}$,  and conversely for any $b$ there exists $\widetilde{b}$ such that this restriction holds true. Moreover, the extended simbol is unique modulo a function in $S^{-\infty} (\mathbb{T}^n \times \mathbb{R}^n)$.  
\end{remark}

\begin{remark}
In \cite{G-P}  it is  considered the  phase space Fourier representation, 
\begin{equation}
\label{pf-T}
b (x,\eta) = F (\widehat{b})  :=  (2\pi)^{-n} \int_{\mathbb{R}^n} \sum_{q \in \mathbb{Z}^n}    \widehat{b} (q,p) e^{ i (\langle p , \eta \rangle + \langle q , x \rangle)}   dp, \quad \quad (q,p) \in \mathbb{Z}^n \times \mathbb{R}^n, \quad (x,\eta)  \in \Bbb T^n \times \Bbb R^n, 
\end{equation}
(in the sense of  distributions) and the operator 
$
U_\hbar (q,p) \psi (x) := e^{i (q \cdot x + \hbar p \cdot q / 2) } \psi (x+\hbar p)
$ 
which is well defined on $L^{2} (\Bbb T^n)$ for any fixed $(q,p) \in \mathbb{Z}^n \times \mathbb{R}^n$. 
In this framework, the Weyl quantization of a simbol $b\in S^m(\mathbb{T}^n\times\mathbb{R}^n)$  is given by
\begin{equation}
\label{Q2}
{\rm Op}_\hbar^w (b) \psi (x)   := (2\pi)^{-n} \int_{\mathbb{R}^n}   \sum_{q \in \mathbb{Z}^n} \widehat{b} (q,p) U_\hbar (q,p) \psi(x)  dp.
\end{equation}
Consequently, the corresponding  Wigner transform and Wigner distribution are
\begin{eqnarray}
\label{med5}
\widehat{W}_\hbar \psi (q,p)  &:=& \langle   \psi ,  U_\hbar (q,p) \psi \rangle_{L^2}. 
\\
\langle \psi, {\rm Op}^w_\hbar (b) \psi \rangle &:=&\int_{\mathbb{R}^n}  \sum_{q \in \mathbb{Z}^n}   \widehat{b} (q,p) \widehat{W}_\hbar \psi (q,p) dp. 
\label{Wigner2}
\end{eqnarray}
In fact, the  Weyl quantizations as in (\ref{weyl}) and (\ref{Q2})  are coinciding (see  Proposition 2.3 in \cite{P-Z}).
\end{remark}

\subsubsection{Composition and Boundedness for Weyl operators}
In the following we recall a result on $L^2(\Bbb T^n)$-boundedness for a class of operators involved in our paper. 

\begin{theorem}[see \cite{G-P}]
\label{Th-Bound0}
Let  ${\rm Op}^w_\hbar (b)$ as in  (\ref{Q2}) with $b \in S^0_{0,0} (\mathbb{T}^n \times \mathbb{R}^n)$. Let $N = n/2 +1$ when $n$ is even, $N = (n+1)/2 +1$ when $n$ is odd. Then, for $\psi \in C^\infty (\mathbb{T}^n)$
\begin{equation}
\|  {\rm Op}_\hbar^w (b)  \psi \|_{L^2(\Bbb T^n)} \le \frac{2^{n+1}}{n+2} \ \frac{  \pi^{(3n-1)/2}}{\Gamma ((n+1)/2)}  \sum_{|\alpha| \le 2N}  \,    \|  \partial_x^\alpha b \|_{L^\infty (\Bbb T^n \times \Bbb R^n)}    \ \| \psi \|_{L^2 (\Bbb T^n)}. 
\end{equation}
\end{theorem}
\noindent
By using standard arguments (such as Hahn-Banach Theorem, see for example \cite{R-S}) the above class of operators  can be extended as  bounded linear operators on $L^2 (\mathbb{T}^n)$.This is the toroidal counterpart of the well known Calderon-Vaillancourt Theorem for Pdo on $\Bbb R^n$ (see for example \cite{Mar}).\\

We devote now our attention to the composition of these toroidal operators (see also  \cite{G-P}, for a similar result involving a smaller class of simbols).
\begin{theorem}
\label{Th-comp}
Let  $\ell,m \in \Bbb R$, $a \in S^\ell (\mathbb{T}^n \times \mathbb{R}^n)$  and  $b \in S^m (\mathbb{T}^n \times \mathbb{R}^n)$. Then, 
\begin{equation}
\label{comp0}
{\rm Op}^w_\hbar (a) \circ {\rm Op}^w_\hbar (b)  =  {\rm Op}^w_\hbar (  a \sharp b    )
\end{equation} 
where $a \sharp b = a \cdot b + O(\hbar)$ in $S^{\ell+m} (\mathbb{T}^n \times \mathbb{R}^n)$.  Moreover,
\begin{equation}
\label{comm1}
[{\rm Op}^w_\hbar (a) , {\rm Op}^w_\hbar (b)]  =  {\rm Op}^w_\hbar (  a \sharp b - b \sharp a  )
\end{equation} 
where the Moyal bracket reads $\{ a , b \}_{ {\rm M} } := a \sharp b - b \sharp a   = - i\hbar  \{  a , b \} + O(\hbar^2) $ in $S^{\ell+m-1} (\mathbb{T}^n \times \mathbb{R}^n)$.
\end{theorem}
\begin{proof}
To begin, we observe that $T_\omega \psi(y) := \psi(2y-\omega)$ can be written as
\begin{equation}
T_\omega \psi(y) = (2\pi)^{-n}\sum_{\kappa \in\mathbb{Z}^n}\int_{\mathbb{T}^n}e^{i\langle (2y-\omega) -z,\kappa \rangle} \psi(z)dz, \quad \forall \ \psi \in C^\infty (\mathbb{T}^n;\Bbb C).
\end{equation}
By Theorem 8.4 in \cite{R-T}, it follows 
\begin{equation}
 {\rm Op}^w_\hbar (b) \psi (x) = (b (X, \frac{\hbar}{2} D) \circ T_{\omega = x}  \psi )(x) = (2\pi)^{-n}\sum_{\kappa \in\mathbb{Z}^n}\int_{\mathbb{T}^n}e^{i\langle x -z,\kappa \rangle}  c(\hbar,x,z,\kappa)  \psi(z)dz
\end{equation} 
with {\it toroidal amplitude} $c(\hbar,\cdot) \in C^\infty (\Bbb T^n_x \times \Bbb T^n_z \times \Bbb Z^n_\kappa)$ such that  
$
| \partial_x^\alpha \partial_z^\gamma    c(\hbar,x,z,\kappa) |   \le  C_{\alpha \gamma} \langle \kappa \rangle^{\ell+m}
$.
In particular,  $c = b(z,\frac{\hbar}{2}\kappa) + O(\hbar)$ in $S^{m} (\Bbb T^n \times \Bbb T^n \times \Bbb Z^n)$.
Now, apply Theorem 4.2 in \cite{R-T}, so that there exists a unique toroidal simbol $\sigma(\hbar, \cdot) \in S^{m} (\Bbb T^n \times \Bbb Z^n)$ such that
\begin{equation}
\label{rep-Op}
 {\rm Op}^w_\hbar (b) \psi (x)  = (2\pi)^{-n}\sum_{\kappa \in\mathbb{Z}^n}\int_{\mathbb{T}^n}e^{i\langle x-y,\kappa \rangle} \sigma (\hbar,y,\kappa)\psi(y)dy.
\end{equation} 
where in particular $\sigma  (\hbar,y,\kappa) = b(y,\frac{\hbar}{2}\kappa) + O(\hbar)$ in  $S^{m} (\Bbb T^n \times \Bbb Z^n)$.
By Theorem 4.3 in \cite{R-T}, it follows the existence of $\widehat{a \sharp b}  (\hbar, \cdot) \in S^{\ell +m} (\Bbb T^n \times \Bbb Z^n)$ such that  
\begin{equation}
{\rm Op}^w_\hbar (a) \circ {\rm Op}^w_\hbar (b) \psi  (x) = (2\pi)^{-n}\sum_{\kappa \in\mathbb{Z}^n}\int_{\mathbb{T}^n}e^{i\langle x-y,\kappa \rangle} \widehat{a \sharp b} (\hbar,y,\kappa)\psi(y)dy
\end{equation}
with $\widehat{a \sharp b}(\hbar,y,\kappa)   = a \cdot b (y,\frac{\hbar}{2}\kappa)  + O(\hbar)$ in $S^{\ell +m} (\Bbb T^n \times \Bbb Z^n)$. Now apply this operator on $T_x^{-1} \circ T_x \psi$, use again Theorems 8.4 and 4.2  in order to get    
\begin{equation}
{\rm Op}^w_\hbar (a) \circ {\rm Op}^w_\hbar (b)  =  {\rm Op}^w_\hbar (  \widetilde{a\sharp b}    )
\end{equation}
where $\widetilde{a \sharp b} (\hbar,y,\kappa)  = a \cdot b (y,\kappa) + O(\hbar)$ in $S^{\ell +m} (\Bbb T^n \times \Bbb Z^n)$.
By Theorem 5.2 in \cite{R-T} we get an euclidean simbol $a \sharp b   \in S^{\ell +m} (\Bbb T^n \times \Bbb R^n)$ which is an extention of  $\widetilde{a \sharp b}$ modulo $S^{-\infty} (\Bbb T^n \times \Bbb R^n)$, and thus  such that 
\begin{equation}
{\rm Op}^w_\hbar (a) \circ {\rm Op}^w_\hbar (b)  =  {\rm Op}^w_\hbar ( a \sharp b  )
\end{equation}
where  $a \sharp b (\hbar,y,\kappa)   = a \cdot b (y,\kappa)  + O(\hbar)$ but now in $S^{\ell +m} (\Bbb T^n \times \Bbb R^n)$. By   looking at the second order asymptotics of the simbols,  it follows $a \sharp b - b \sharp a   = - i\hbar  \{  a , b \} + O(\hbar^2)$ in $S^{\ell+m-1} (\mathbb{T}^n \times \mathbb{R}^n)$, and this gives (\ref{comm1}).
\end{proof}

\subsubsection{Wigner measures}\label{wig1}
To begin, let us recall that in the framework of the usual Weyl quantization  on $\Bbb R^n$ it can be considered the following space of test functions (see for example \cite{A-F-G},  \cite{L-P})
%it  is sufficent to show that  $W_\hbar \psi_\hbar (x,\xi)$ is  bounded (uniformely with respect to $\hbar$) in $\mathcal{A}^\prime$ where 
\begin{equation}
\mathcal{A} :=  \{ \varphi \in C_0 (\Bbb R_x^n \times \Bbb R_\xi^n) \ | \ \| \varphi \|_\mathcal{A}  :=  \int_{\Bbb R^n} \sup_{x \in \Bbb R^n} | \mathcal{F}_\xi \varphi (x,z)| 	\  dz < + \infty \}
\end{equation}
where $C_0 (\Bbb R_x^n \times \Bbb R_\xi^n)$ denotes the set of continuous functions tending to zero at infinity, and $\mathcal{F}_\xi$ is the usual Fourier transform in the frequency variables, i.e. $ \mathcal{F}_\xi \varphi (x,z) := \int_{\Bbb R^n} e^{- i \xi \cdot z} \varphi (x,\xi) d\xi$. In particular, $\mathcal{A}$ is a Banach space and it is a dense subset of  $C_0 (\Bbb R_x^n \times \Bbb R_\xi^n)$. Hence, its dual space $\mathcal{A}^\prime$ contains $C_0^\prime  (\Bbb R_x^n \times \Bbb R_\xi^n) = \mathcal{M} (\Bbb R_x^n \times \Bbb R_\xi^n)$ the space of not necessarily nonnegative Radon measures on $\Bbb R^{2n}$  of finite mass.
As shown in Proposition III.1 of \cite{L-P}, it holds the inequality
\begin{equation}
\label{estW18}
\Big| \, \int_{\Bbb  R^n}  \int_{\Bbb  R^n}  W_\hbar \psi_\hbar (x,\xi)  \varphi(x,\xi) dxd\xi \,  \Big| \le (2\pi)^{-n}   \| \varphi \|_\mathcal{A}  \cdot \| \psi_\hbar  \|_{L^2},
\end{equation}
and hence for any family of wave functions such that $\|  \psi_\hbar \|_{L^2 (\Bbb R^n)} \le C$  there exists a sequence $\hbar_j \longrightarrow 0^+$ as $j \longrightarrow + \infty$ such that $ W_{\hbar_j} \psi_{\hbar_j}$ is converging  in  $\mathcal{A}^\prime$ to some $W \in \mathcal{A}^\prime$ (thanks Banach--Alaoglu theorem). Moreover, through the use of Husimi transform, it can be proved that  in fact  any such limit $W \in \mathcal{A}^\prime$  fulfills also $W \in \mathcal{M}^+ (\Bbb R_x^n \times \Bbb R_\xi^n)$, i.e. positive Radon measures of finite mass.

We underline that there is an estimate analogous to (\ref{estW18}) for our toroidal framework which takes the form 
\begin{equation}
\label{estW76}
\Big| \sum_{ \eta \in \frac{\hbar}{2} \mathbb{Z}^n} \int_{\mathbb{T}^n}  W_\hbar \psi_\hbar (x,\eta) g(x,\eta) dx \Big| \le  (2\pi)^{-n}  \sup_{(x,\eta) \in \Bbb T^n \times \Bbb R^n} |  g (x,\eta)|   \cdot  \| \psi_\hbar  \|_{L^2}
\end{equation}
for all continuous bounded functions $g :\Bbb R^{2n} \longrightarrow \Bbb R$. 
Indeed, we observe  that  for states $\psi_\hbar \in L^2 (\Bbb T^n)$, by writing the Fourier series $\psi_\hbar (x)= \sum_{\alpha \in \Bbb Z^n} \widehat{\psi}_{\hbar,\alpha} \ e^{i \langle x , \alpha \rangle}$ we have
\begin{itemize}
\item[(i)] $\displaystyle{\sum_{\eta \in \frac{\hbar}{2}  \Bbb Z^n}     W_\hbar \psi_\hbar (x,\eta) = |\psi_\hbar (x)|^2}$, 
\item[(ii)] $
\displaystyle{  (2\pi)^{-n}  \int_{\Bbb T^n}   W_\hbar \psi_\hbar (x,\eta)  dx  = \left\{ \begin{array}{lll}   |\widehat{\psi}_{\hbar,\alpha}|^2  \ \ \text{\rm when} \ \eta = \hbar \alpha,\quad   \ \alpha \in \Bbb Z^n,   \\ 
0 \quad \quad \quad {\rm otherwise}. 
\end{array}    
\right.}
$ 
\end{itemize} 
Hence, by property (ii) it follows the estimate (\ref{estW76}).\\ 
In view of the above observations, we can now introduce the following
\begin{definition}[{\bf Test  functions}]
Let $C_0 (\Bbb T_x^n \times \Bbb R_\eta^n)$ be the set of real valued continuous functions on $\Bbb T_x^n \times \Bbb R_\eta^n$ tending to zero at infinity in $\eta$-variables. We consider the subset of those $\phi \in C_0 (\Bbb T_x^n \times \Bbb R_\eta^n)$ that admit  the phase space Fourier representation $\phi = F (\widehat{\phi})$ as in (\ref{pf-T}) for some  compactly supported $\widehat{\phi} : \Bbb Z^n \times \Bbb R^n \longrightarrow \Bbb C$.
We define the set
\label{def-wc}
\begin{equation}
A := \overline{  \Big\{ \phi \in C_0 (\Bbb T_x^n \times \Bbb R_\eta^n) \ | \    {\rm supp}(\widehat{\phi}) \ {\rm is \ compact}     \Big\} }^{ \, L^\infty }.
\end{equation}
Notice that $A$ is a closed linear subset of $L^{\infty} (\Bbb T_x^n \times \Bbb R_\eta^n)$ hence it becomes a Banach space when equipped by the $L^\infty$-norm. 
%We recall that  $C_0^\infty (\Bbb T_x^n \times \Bbb R_\eta^n)$ is dense in $C_0 (\Bbb T_x^n \times \Bbb R_\eta^n)$ with respect to $L^\infty$ - norm, and moreover 
We also underline that for any fixed $\phi \in C_0 (\Bbb T_x^n \times \Bbb R_\eta^n)$ such that  ${\rm supp}(\widehat{\phi})$ is  compact then $\phi$ is necessarily a $C^\infty$ function rapidly decreasing in $\eta$-variables, and hence we can directly deal with the set of $C^\infty$ functions vanishing at infinity in the $\eta$-variables $C_0^\infty (\Bbb T_x^n \times \Bbb R_\eta^n)$. Thus, we can write 
\begin{equation}
\label{setA2}
A = \overline{  \Big\{ \phi \in C_0^\infty (\Bbb T_x^n \times \Bbb R_\eta^n) \ | \    {\rm supp}(\widehat{\phi}) \ {\rm is \ compact}     \Big\} }^{ \, L^\infty }.
\end{equation}
Moreover, we easily see that $A \subset C_b ( \Bbb T^n \times \Bbb R^n)$.
\end{definition}

\noindent
 We are now in the position to provide the
\begin{definition}[{\bf Wigner measures}]
\label{def-wc1}
Let us fix  $\{ \psi_\hbar \}_{0 < \hbar \le 1}  \in L^2 (\mathbb{T}^n)$  with $\| \psi_\hbar  \|_{L^2} \le C$ $\forall 0 <\hbar \le 1$. 
We say that  $dw \in \mathcal{M} (\Bbb R_x^n \times \Bbb R_\eta^n)$ is the  Wigner measure of the sequence  $\{ \psi_\hbar \}_{0 < \hbar \le 1}$  if  $\forall \phi \in A$    
\begin{equation}
\label{med}
\sum_{ \eta \in \frac{\hbar}{2} \mathbb{Z}^n} \int_{\mathbb{T}^n} \phi(x,\eta)  W_\hbar \psi_\hbar (x,\eta) dx   \longrightarrow   \int_{\Bbb T^n \times \Bbb R^n}   \phi(x,\eta) dw (x,\eta) \quad \quad 
\end{equation}
for some sequence $\hbar = \hbar_j \ \longrightarrow 0^+$  as $j \longrightarrow + \infty$.
\end{definition}
\noindent

\begin{remark} 
The Wigner transform of $\psi_\hbar \in C^\infty (\Bbb T^n)$ 
\begin{equation}
W_\hbar \psi_\hbar (x,\eta) := (2\pi)^{-n}\int_{\Bbb{T}^n} e^{2 \frac{i}{\hbar} \langle z,\eta\rangle}\psi_\hbar (x-z) \psi^\star_\hbar (x+z)dz, \quad \eta \in \frac{\hbar}{2} \Bbb Z^n,
\label{WignerT23}
\end{equation}
can be rewritten, when acting on test functions $\phi$, as 
\begin{eqnarray}
\label{act-W}
\sum_{\eta \in \frac{\hbar}{2} \mathbb{Z}^n} \int_{\mathbb{T}^n} \phi(x,\eta) W_\hbar \psi_\hbar (x,\eta)dx &=&  \sum_{\kappa \in  \mathbb{Z}^n} \int_{\mathbb{T}^n} 
\phi \Big(x, \frac{2}{\hbar} \kappa\Big) W_\hbar \psi_\hbar \Big(x,\frac{2}{\hbar} \kappa \Big)dx,
\\
W_\hbar \psi_\hbar \Big(x,\frac{2}{\hbar} \kappa \Big) &=& (2\pi)^{-n}\int_{\Bbb{T}^n} e^{i \langle z,\kappa\rangle}\psi_\hbar (x-z) \psi^\star_\hbar (x+z)dz, \quad  \kappa \in  \mathbb{Z}^n. 
\end{eqnarray}
Thus, we notice the $2\pi \Bbb Z^n$ - periodicity properties
\begin{eqnarray}
W_\hbar \psi_\hbar \Big(x,\frac{2}{\hbar} (\kappa + 2\pi \alpha) \Big) &=&  W_\hbar \psi_\hbar \Big(x,\frac{2}{\hbar} \kappa   \Big)  \quad \forall \alpha \in \Bbb Z^n,
\\
W_\hbar \psi_\hbar \Big(x + 2\pi \alpha ,\frac{2}{\hbar} \kappa  \Big) &=&  W_\hbar \psi_\hbar \Big(x,\frac{2}{\hbar} \kappa   \Big)  \quad \forall \alpha \in \Bbb Z^n.
\end{eqnarray}
From (\ref{WignerT23}) we also easily obtain the  estimate
\begin{equation}
\label{sup-W}
\sup_{\eta \in \frac{\hbar}{2} \Bbb Z^n}  \sup_{x \in \Bbb T^n} |W_\hbar \psi_\hbar (x,\eta)| \le (2\pi)^{-n} \| \psi_\hbar \|_{L^2}^2.   
\end{equation}
Notice that if $\eta \notin \frac{\hbar}{2} \Bbb Z^n$ then (\ref{WignerT23}) is not defined, since we are computing the integral over the torus and thus we need the $2\pi \Bbb Z^n$ periodicity with respect to $x$-variables of the function within the integral. For this reason, we cannot regard $W_\hbar \psi_\hbar (x,\eta)$ as a wellposed function belonging to $L^\infty (\Bbb T^{n}_x \times \Bbb R^n_\eta)$ even if we exhibited the estimate (\ref{sup-W}). This is one of the main differences with the Weyl quantization on $\Bbb R^{n}$ where the Wigner transform $W_\hbar \psi_\hbar (x,\xi)$, when $\psi_\hbar \in L^2 (\Bbb R^n)$,  is a well defined function in $L^\infty (\Bbb R^{n}_x \times \Bbb R^n_\xi)$ for any $\hbar >0$.
\end{remark}
In the toroidal framework of this paper, under the general assumption  $\| \psi_\hbar  \|_{L^2} \le C$ with $C >0$ independent of $\hbar$ we obtain semiclassical limits in $A^\prime$ (see Lemma \ref{t-depW}) and for suitably defined wave functions (as for example the WKB ones shown in Section \ref{SEC-wkb}) we can recover semiclassical limits as  probability measures on $\Bbb T^n \times \Bbb R^n$. 
\begin{lemma}
\label{t-depW}
Let $\{ \psi_\hbar (t) \}_{0 < \hbar \le 1}$ a sequence in $C([-T,T]; L^2 (\Bbb T^n))$ such that $\| \psi_\hbar (t)  \|_{L^2} \le C_T$ for all $t \in [-T,T]$ and $0 < \hbar \le 1$.  Then, there is a sequence $\hbar_j \longrightarrow 0^+$ as $j \longrightarrow + \infty$ such that $W_{\hbar_j} \psi_{\hbar_j} \rightharpoonup W$ in $L^\infty ([-T,+T];A^\prime)$ with $A$ as in Def \ref{def-wc}.
\end{lemma}
\begin{proof}
%The $L^\infty ([-T,+T];A^\prime)$ is the dual of the separable space  $L^1 ([-T,+T] ; A)$ and 
Since we are assuming $\psi_\hbar \in C([-T,T]; L^2 (\Bbb T^n))$ with $\| \psi_\hbar (t)  \|_{L^2} \le C_T$ then the estimate (\ref{estW76}) implies that for $0 < \hbar \le 1$, the family $W_{\hbar} \psi_{\hbar}$ is bounded in $L^\infty ([-T,+T];A^\prime)$. However, $L^\infty ([-T,+T];A^\prime)$ is the dual of the separable space  $L^1 ([-T,+T] ; A)$ and hence the application of the Banach-Alaoglu Theorem provides the existence of a converging subsequence $W_{\hbar_j} \psi_{\hbar_j} \rightharpoonup W$ in $L^\infty ([-T,+T];A^\prime)$.
\end{proof}

We devote now our attention on the following (locally finite) Borel complex measure on $\Bbb T^n \times \Bbb R^n$. 
Let $\mathcal{X}_\Omega$ be the characteristic function of a Borel set $\Omega \subseteq \Bbb T^n \times \Bbb R^n$, we define
\begin{equation}
\label{PO}
\mathbb{P}_\hbar (\Omega) := \sum_{ \eta \in \frac{\hbar}{2} \mathbb{Z}^n} \int_{\mathbb{T}^n} \mathcal{X}_\Omega (x,\eta)  W_\hbar \psi_\hbar (x,\eta) dx.  
\end{equation}
which is a (complex valued) countably additive set function on the Borel sigma algebra of $\Bbb T^n \times \Bbb R^n$. 
In particular, we notice that if $\| \psi_\hbar \|_{L^2} = 1$ then $|\mathbb{P}_\hbar (\Omega) | \le 1$ for all $\Omega \subseteq \Bbb T^n \times \Bbb R^n$ and $|\mathbb{P}_\hbar ( \Bbb T^n \times \Bbb R^n) | = 1$. As usual, we say that  $\mathbb{P}_\hbar$ is weak (i.e. narrow) convergent to a Borel complex measure $\mathbb{P}$ if $\forall f \in C_b ( \Bbb T^n \times \Bbb R^n)$ it holds 
\begin{equation}
\label{NC}
 \int_{\Bbb T^n \times \Bbb R^n} f(x,\eta) d\mathbb{P}_{\hbar} (x,\eta)  \longrightarrow   \int_{\Bbb T^n \times \Bbb R^n} f(x,\eta) d\mathbb{P} (x,\eta)
\end{equation}
as $\hbar \ \longrightarrow 0^+$. 
%On the other hand, take $\mathcal{X}_\mu \in C^\infty_0 (\Bbb T^n \times \Bbb R^n)$ with $\mu \in \Bbb Z^n$ such that $1 = \sum_\mu \mathcal{X}_\mu (x,\eta)$ (see for example Lemma 2.8.2 in \cite{Mar}) thus $ \int_{\Bbb T^n \times \Bbb R^n} f(x,\eta) d\mathbb{P}_{\hbar} (x,\eta) = \sum_\mu  \int_{\Bbb T^n \times \Bbb R^n} f(x,\eta) \mathcal{X}_\mu (x,\eta) d\mathbb{P}_{\hbar} (x,\eta) $. Moreover, for suitable  $c_{ \mu \alpha} \in \Bbb R$ (depending on $f$) and characteristic functions $\mathcal{X}_{\mu \alpha} \in C^\infty_0 (\Bbb T^n \times \Bbb R^n)$ we have $f \mathcal{X}_\mu = \lim_{\alpha \rightarrow + \infty}  c_{ \mu \alpha} \mathcal{X}_{\mu \alpha} $ in $L^\infty$-norm. 
In fact, since $f \in C_b ( \Bbb T^n \times \Bbb R^n)$,  it holds
\begin{equation}
\label{NC2}
 \int_{\Bbb T^n \times \Bbb R^n} f(x,\eta) d\mathbb{P}_{\hbar} (x,\eta) = \sum_{ \eta \in \frac{\hbar}{2} \mathbb{Z}^n} \int_{\mathbb{T}^n} f (x,\eta)  W_{\hbar} \psi_{\hbar} (x,\eta) dx .
\end{equation}

\begin{definition}
The family of (complex Borel) measures $\{ \mathbb{P}_\hbar \}_{0 < \hbar \le 1}$ on the probability space $\Bbb T^n \times \Bbb R^n$ (equipped with the Borel sigma algebra) is called {\it tight} if 
\begin{equation}
\label{tight-C}
\lim_{R \rightarrow + \infty}  \, \sup_{0 < \hbar \le 1} \,  \int_{ \Bbb T^n \times \{ \Bbb R^{n} \backslash B_R \} }  d\mathbb{P}_{\hbar} (x,\eta)  = 0.
\end{equation}
Thanks to a well-known Prokhorov's Theorem, the set of measures  $\{ \mathbb{P}_\hbar \}_{0 < \hbar \le 1}$ is relatively compact with respect to the weak  topology if and only if  is tight. Notice that  the condition (\ref{tight-C}) reads equivalently as $\lim_{R \rightarrow + \infty}  \sup_{0 < \hbar \le 1} \mathbb{P}_{\hbar} (\Bbb T^n \times \{ \Bbb R^{n} \backslash B_R \}) = 0  $.
\end{definition}

\begin{remark}
When $\mathbb{P}_\hbar = \mathbb{P}_\hbar^\pm$ is associated to the class of WKB wave functions  $\varphi_\hbar^{\pm}$ described in Section \ref{SEC-wkb}, we will directly prove the weak convergence (with test functions in A)  to some meaningful  probability measures of  monokinetic type (see Theorem \ref{TH4}). On the other hand,  within Lemma \ref{L-T1} we will also prove that such measures $\mathbb{P}_\hbar^\pm$ fulfill the tightness condition (\ref{tight-C}), and in this way we can apply the next result on time propagation of tightness.  This ensures the existence of the Wigner probability measure associated to the solution of the Schr\"odinger equation, and its coincidence with the solution of the underlying classical continuity equation,  see Theorem \ref{th51} and  Proposition \ref{prop53}.  
\end{remark}

\begin{proposition}[{\bf Propagation of tightness}]
\label{Prop-T}
Let $H = \frac{1}{2} |\eta|^2 + V(x)$ with $V \in C^\infty (\Bbb T^n)$, $\psi_\hbar \in L^{2} (\Bbb T^n)$ be such that $\| \psi_\hbar   \|_{L^2} \le C$ for all $0 < \hbar \le 1$. Assume that $\mathbb{P}_\hbar$ as in (\ref{PO}) is tight. 
Define $\psi_\hbar (t) := e^{ -\frac{i}{\hbar} {\rm Op}_\hbar (H) t} \psi_\hbar$. Then, $\mathbb{P}_\hbar (t)$ is tight for any $t \in \Bbb R$.
\end{proposition}
\begin{proof}
Let $Y \in C^{\infty} (\Bbb R^n_\eta ;[0,1])$ be such that $Y(\eta) = 1$ on $|\eta| >1$ and $Y(\eta) = 0$ on $|\eta| < 1/2$; for $R >0$ define $Y_R (\eta) := Y(\eta/R)$. Then, 
$|\nabla_\eta Y| \le C/R$ and $|\nabla_\eta^2 Y| \le C/R^2$ for some $C >0$.  In fact, we can regard $Y \in C^{\infty}_b (\Bbb T^n_x \times \Bbb R^n_\eta ;[0,1])$.
\begin{equation}
\frac{d}{ds} \langle \psi_\hbar (s), {\rm Op}_\hbar (Y_R) \psi_\hbar (s) \rangle_{L^2} = \frac{i}{\hbar} \langle  \psi_\hbar (s), [{\rm Op}_\hbar (Y_R) , {\rm Op}_\hbar (H)   ] \psi_\hbar (s)     \rangle_{L^2}. 
\end{equation}
Recalling Theorem \ref{Th-comp}, the commutator reads $[{\rm Op}_\hbar (Y_R) , {\rm Op}_\hbar (H)   ] = {\rm Op}_\hbar (\{ Y_R , H  \}_M)$ where the Moyal bracket has the asymptotics $\{ Y_R , H  \}_M =  -i \hbar \{ Y_R , H  \} + D_\hbar$ in $S^2 (\Bbb T^n \times \Bbb R^n)$ where the remainder $D_\hbar \simeq O(\hbar^2)$ and involves the second order derivatives of $Y_R $ and $H$. But $|\partial_x^\alpha \partial_\eta^\beta H(z)| \le c_1$ and $|\partial_x^\alpha \partial_\eta^\beta Y_R (z)| \le c_2/R^2$ for $|\alpha + \beta| = 2$;  hence $| D_\hbar | \simeq R^{-2}$ as $R  \longrightarrow + \infty$ (uniformly on $\hbar$). Moreover $\{ Y_R , H  \}(z) = \partial_x Y_R \partial_\eta H -  \partial_\eta Y_R \partial_x H  = -  \partial_\eta Y_R \partial_x H $ hence  $| \{ Y_R , H  \}(z)|  \le c_3 /R$. By recalling the $L^2$ - boundedness of the Weyl operators with simbols in $S^0_{0,0}(\Bbb T^n \times \Bbb R^n)$ as shown in Theorem \ref{Th-Bound0} and using the assumption $\| \psi_\hbar   \|_{L^2} \le C$, we deduce that
\begin{equation}
\Big| \frac{d}{ds} \langle \psi_\hbar (s), {\rm Op}_\hbar (Y_R) \psi_\hbar (s) \rangle_{L^2} \Big| \le K \cdot R^{-1} 
\end{equation}
for some $K > 0$ independent on $\hbar$ and $t$. Thus
\begin{equation}
\langle \psi_\hbar (t), {\rm Op}_\hbar (Y_R) \psi_\hbar (t) \rangle_{L^2}  =  \langle \psi_\hbar (0), {\rm Op}_\hbar (Y_R) \psi_\hbar (0) \rangle_{L^2}  +   \int_0^t   \frac{d}{ds} \langle \psi_\hbar (s), {\rm Op}_\hbar (Y_R) \psi_\hbar (s) \rangle_{L^2} ds 
\end{equation}
and 
\begin{eqnarray}
|\langle \psi_\hbar (t), {\rm Op}_\hbar (Y_R) \psi_\hbar (t) \rangle_{L^2} | &\le&  | \langle \psi_\hbar (0), {\rm Op}_\hbar (Y_R) \psi_\hbar (0) \rangle_{L^2}|  + \Big|  \int_0^t   \frac{d}{ds} \langle \psi_\hbar (s), {\rm Op}_\hbar (Y_R) \psi_\hbar (s) \rangle_{L^2} ds \Big| 
\nonumber
\\
&\le&  | \langle \psi_\hbar (0), {\rm Op}_\hbar (Y_R) \psi_\hbar (0) \rangle_{L^2}|  + t \ K \cdot R^{-1} 
\end{eqnarray}
Notice that, from the property (ii) of $W_\hbar \psi_\hbar$, it follows
\begin{eqnarray}
\langle \psi_\hbar , {\rm Op}_\hbar (Y_R) \psi_\hbar  \rangle_{L^2} &=&
  \sum_{ \eta \in \frac{\hbar}{2} \mathbb{Z}^n} \int_{\mathbb{T}^n}  W_\hbar \psi_\hbar (x,\eta) Y_R (\eta) dx 
  \\
 &=&  \sum_{ \eta \in \frac{\hbar}{2} \mathbb{Z}^n} Y_R (\eta)  \int_{\mathbb{T}^n}  W_\hbar \psi_\hbar (x,\eta)  dx = 
  \sum_{ \alpha \in \mathbb{Z}^n}  Y_R (\hbar \alpha)  |\widehat{\psi}_{\hbar,\alpha}|^2, 
\end{eqnarray}
thus any term of the series is non negative. The same holds true for 
\begin{eqnarray}
\mathbb{P}_\hbar  (\Bbb T^n \times U) &=&
 \sum_{ \eta \in \frac{\hbar}{2} \mathbb{Z}^n} \int_{\mathbb{T}^n}  W_\hbar \psi_\hbar (x,\eta) \mathcal{X}_U (\eta) dx 
 \\
&=&  \sum_{ \eta \in \frac{\hbar}{2} \mathbb{Z}^n} \mathcal{X}_U  (\eta)  \int_{\mathbb{T}^n}  W_\hbar \psi_\hbar (x,\eta)  dx = 
\sum_{ \alpha \in \mathbb{Z}^n}  \mathcal{X}_U (\hbar \alpha)  |\widehat{\psi}_{\hbar,\alpha}|^2, 
\label{P-U}
\end{eqnarray}
where $U$ is any Borel set in $\Bbb R^n$.\\
By defining $M_R :=   \Bbb T^n \times \{ \Bbb R^{n} \backslash B_R \}$, and recalling that $Y_R (\eta) = 0$ for $|\eta| < R/2$ whereas $Y_R(\eta) = 1$ for $|\eta| > R$, we can write
\begin{eqnarray}
\mathbb{P}_\hbar (t) (M_R)  &\le& \langle \psi_\hbar (t), {\rm Op}_\hbar (Y_R) \psi_\hbar (t) \rangle_{L^2} 
\\
&\le&  \langle \psi_\hbar (0), {\rm Op}_\hbar (Y_R) \psi_\hbar (0) \rangle_{L^2}  + t \ K \cdot R^{-1} 
\\
&\le&  \mathbb{P}_\hbar  (M_{R/2})   + t \ K \cdot R^{-1} 
\end{eqnarray} 
and hence (recalling the tightness assumption on $\mathbb{P}_\hbar$) 
\begin{equation}
\lim_{R \rightarrow + \infty}   \sup_{0 < \hbar \le 1} \mathbb{P}_\hbar (t) (M_R)  = 0.
\end{equation}
\end{proof}

\subsection{A quick review of weak KAM theory and Aubry-Mather theory}\label{kam1}
\subsubsection{Weak solutions of Hamilton-Jacobi equation}
\label{sub-wHJ}
The  {\it weak KAM theory} deals with the existence of Lipschitz continuous solutions of the stationary Hamilton-Jacobi equation
\begin{equation}
\label{def-eff0}
H(x,P+\nabla_x v(P,x))  =  \bar{H}(P), \quad P \in \Bbb R^n, 
\end{equation}
for Tonelli Hamiltonians  $H \in C^\infty (\mathbb{T}^n \times \mathbb{R}^n;\mathbb{R})$, that is to say, for functions $H$ such that $\eta \mapsto H(x,\eta)$ 
is strictly convex and uniformly superlinear in the fibers of the canonical projection $\pi : \mathbb{T}^n \times \mathbb{R}^n \longrightarrow \mathbb{T}^n$.  
The function $\bar{H}(P)$ is called the {\it effective Hamiltonian} and, as showed in \cite{C-I-P} (see also \cite{E-G}),  
it can be expressed by the inf-sup formula
\begin{equation}
\label{def-eff}
\bar{H}(P) =  \inf_{v \in C^\infty  (\mathbb{T}^n;\mathbb{R})} \  \sup_{x \in \mathbb{T}^n}  \ H(x,P+\nabla_x v(x)) 
\end{equation}
which is a convex function of $P \in \Bbb R^n$ (hence continuous).
The Lax-Oleinik semigroup of negative and positive type is defined as
$$
T_t^{\mp} u (x) :=  \inf_{\gamma} \left\{ u(\gamma(0))  \pm  \int_0^t L(\gamma(s),\dot{\gamma}(s))  - P \cdot \dot{\gamma}(s)  \  ds \right\},
$$ 
where the infimum is taken over all absolutely continuous curves $\gamma : [0,t] \longrightarrow \mathbb{T}^n$ such that $\gamma(t)=x$. 
A  function $v_{-} \in C^{0,1}( \mathbb{T}^n ; \Bbb R)$   is said to be a {\it weak KAM solution of negative type} for (\ref{def-eff0}) if $\forall$ 
$t\ge0$
\begin{equation}
\label{back-}
T_t^{-} v_{-}  = v_{-}  - t \, \bar{H}(P) ,  
\end{equation}
whereas  it is said to be a {\it weak KAM solution of positve type} if  $\forall$ $t\ge0$ 
\begin{equation}
\label{back+}
T_t^{+} v_{+}  = v_{+} + t \, \bar{H}(P). 
\end{equation}
As  a consequence, for any weak KAM solution it holds 
\begin{equation}
\label{inc-G}
\overline{{\rm Graph} (P + \nabla_x v_\pm (P,\cdot))}  \subset  \{ (x,\eta) \in \Bbb T^n \times \Bbb R^n \ | \ H(x,\eta) = \bar{H}(P) \}
\end{equation}
Geometrically, equations (\ref{back-}) and (\ref{back+}) imply also that we are looking at functions for which the graphs are invariant under the 
backward (resp. forward) Euler-Lagrange flow, namely 
\begin{equation}
\varphi_H^t \Big( {\rm Graph}(P + \nabla_x v_- (P,\cdot))  \Big) \subseteq {\rm Graph}(P + \nabla_x v_- (P,\cdot)) \quad \forall t \le 0
\end{equation}
and 
\begin{equation}
\varphi_H^t \Big( {\rm Graph}(P + \nabla_x v_+ (P,\cdot))  \Big) \subseteq {\rm Graph}(P + \nabla_x v_+ (P,\cdot)) \quad \forall t \ge 0
\end{equation}
see Theorems 4.13.2 and 4.13.3 in \cite{F}. Moreover, it is proved that the maps $x \longmapsto (x,P + \nabla_x v_\pm (P,x))$ are continuous on  ${\rm  dom}(\nabla_x v_\pm)$.
As showed within Th. 7.6.2 of \cite{F}, all the Lipschitz continuous weak KAM solutions of negative type coincide with the so-called {\it viscosity solutions} in the sense of Crandall-Lions  \cite{C-L}.  
%We remind  that a weak KAM $v_{-}$ is said to be conjugated to $v_{+}$ if $v_{-} = v_{+}$ on  the projected Mather set $\mathcal{M}_P$ (see the next subsection). 

\subsubsection{Mather measures}
\label{sec-M}
The Aubry-Mather theory proves the existence of  invariant and Action-minimizing measures as well as  invariant and Action-minimizing sets in the 
phase space. Here  we recall only those results which we are going to use in what follows, and for an exahustive treatment we address the reader to \cite{Ma1}, \cite{M1}, \cite{So}.\\
Recall that a compactly supported Borel probability measure $d\mu$  on the tangent bundle $T \mathbb{T}^n$ is called {\it invariant}  with respect to the Lagrangian flow  
$\phi^t : \mathbb{T}^n \times \mathbb{R}^n \longrightarrow \mathbb{T}^n \times \mathbb{R}^n$  related to a Lagrangian function $L(x,\xi)$, which we suppose to be
Legendre-related to a Tonelli Hamiltonian $H(x,p)$, if 
\begin{equation*}
\int_{\mathbb{T}^n \times \mathbb{R}^n}  f (\phi^t (x,\xi))   d \mu (x,\xi) =  \int_{\mathbb{T}^n \times \mathbb{R}^n}  f (x,\xi)   d \mu (x,\xi),
\end{equation*}
for all   $t \in \mathbb{R}$ and all $f \in C^\infty_0 (\mathbb{T}^n \times \mathbb{R}^n;\mathbb{R}$. Recall also that a Borel probability measure $d\mu$ is said to be {\it closed} if  
for every $ g \in C^\infty (\mathbb{T}^n;\mathbb{R})$ one has
\begin{equation*}
\int_{\mathbb{T}^n \times \mathbb{R}^n} \nabla_ x g (x)  \cdot \xi \  d \mu (x,\xi) =  0.
\end{equation*}
One says that an  invariant   compactly supported Borel probability measure $d\mu_P$ is a {\it Mather measure} if it satisfies the Mather {\it $P$-minimal problem} 
for all $P \in \mathbb{R}^n$, that is,
\begin{equation*}
\label{mather-P}
\int_{\mathbb{T}^n \times \mathbb{R}^{n}}  \bigl( \, L (x,\xi) - P  \cdot \xi  \, \bigr) \ d\mu_P (x,\xi)  =  
\inf_{d\mu}   \int_{\mathbb{T}^n \times \mathbb{R}^{n}} \bigl( \, L (x,\xi) - P  \cdot \xi  \, \bigr) \ d\mu(x,\xi),   
\end{equation*}
where the infimum is taken over all invariant compactly supported Borel probability measures $d\mu$. 
Moreover, the miminizing value of the Action is related to the effective Hamiltonian as
\begin{equation*}
\label{mather-H}
- \bar{H}(P)  =   \int_{\mathbb{T}^n \times \mathbb{R}^{n}} \bigl( \, L (x,\xi) - P  \cdot \xi \, \bigr) \ d\mu_P (x,\xi).
\end{equation*}
It has been also proved that the Mather measures of a Tonelli-Lagrangian are those which minimize the action in the class of all (compactly supported) closed measures (see for example \cite{B}). 
This fact will be useful in the proof of  Theorem 1.2. 
As for the Mather set, it involves the supports of all Mather's measures, and is defined to be
\begin{equation}
\label{def-M}
\widetilde{\mathcal{M}}_P :=   \overline{ \bigcup_{d\mu_P} {\rm supp} \ d\mu_P }.
\end{equation}
We recall that Mather proved in \cite{M1} that the set $\widetilde{\mathcal{M}}_P$ is not empty, compact and  Lipschitz graphs above $\Bbb T^n$, namely the restriction of $\pi : \Bbb T^n \times \Bbb R^n \rightarrow \Bbb T^n$ to $\widetilde{\mathcal{M}}_P$ is an injective map and 
$
\pi^{-1} : \pi (\widetilde{\mathcal{M}}_P) \rightarrow \widetilde{\mathcal{M}}_P
$
is Lipschitz. The projected Mather set  $\pi (\widetilde{\mathcal{M}}_P)$  is denoted by $\mathcal{M}_P$.\\
By following the Remark 4.11 in \cite{So}, one can take a countably dense set of Mather measures $\{ d\mu_{j,P} \}_{j \in \Bbb N}$ such that 
\begin{equation}
d\bar{\mu}_P := \sum_{j \in \Bbb N} d\mu_{j,P}
\end{equation}
is a Mather measure with full support on the Mather set $\widetilde{\mathcal{M}}_P$. For any fixed Mather measure $d\mu_P$, we denote by 
\begin{equation}
\label{def-dwP}
dw_P := \mathcal{L}_\star (d\mu_P), \quad \quad d \sigma_P :=  \pi_\star (d w_P) = \pi_\star (d\mu_P), 
\end{equation} 
the push forward by the Legendre transform  $\mathcal{L} (x,\xi) = (x,\nabla_\xi L(x,\xi))$ and by the canonical projection $\pi(x,\eta)=x$.

\subsubsection{Aubry sets}
\label{sec-Au}
As for the definition of the Aubry sets $\widetilde{\mathcal{A}}_P$ (in the tangent bundle of a manifold) involving regular $P$-minimizers we refer to \cite{F}; we recall here that its Legendre transform
can be given by
\begin{equation}
\label{def-aubry}
\mathcal{A}^*_P =  \bigcap_{ v \in S^{\mp}_P}  \Big\{  (x, P + \nabla_x v (P,x)) \ | \  x \in \mathbb{T}^n \ {\rm s.t.} \ \exists \  \nabla_x v(P,x)  \Big\}   
\end{equation}
where the intersection is taken over all Lipschitz continuous weak KAM solutions $S^{\mp}_P$ of negative (resp. positive) type of the Hamilton-Jacobi  equation  (\ref{def-eff0}).
This set is invariant under the Hamiltonian dynamics and one has the meaningful inclusion
\begin{equation}
\label{inc-MA}
\mathcal{M}_P^\star :=  \mathcal{L} (\widetilde{\mathcal{M}}_P)  \subseteq \mathcal{A}_P^* .
\end{equation}
The $\mathcal{A}_P^\star$ is compact,  the restriction of $\pi : \Bbb T^n \times \Bbb R^n \rightarrow \Bbb T^n$ to $\mathcal{A}_P^\star$ is an injective map and 
$
\pi^{-1} : \pi (\mathcal{A}_P^\star) \rightarrow \mathcal{A}_P^\star
$
is Lipschitz (see \cite{F}, \cite{So}).

\section{The dynamics of the Wigner transform on the  torus}\label{wtd1}   \markboth{The dynamics of the Wigner transform on the torus}{}
\subsection{The Schr\"odinger equation on the torus}
Let us consider the classical Hamiltonian $H=\frac{1}{2} |\eta|^2 + V(x)$,
with $V \in C^\infty (\Bbb T^n; \Bbb R)$. Thus we have $H \in S^2 (\mathbb{T}^n \times \mathbb{R}^n)$, namely the simbol class described in (\ref{symb00}) with $m=2$. 
We now consider the  Schr\"odinger equation:
\begin{eqnarray}
\label{eqSch} 
i \hbar \partial_t \psi_\hbar (t,x) &=&  {\rm Op}^w_\hbar (H) \psi_\hbar (t,x)
\\
\psi_\hbar (0,x) &=& \varphi_{\hbar}  (x)
\nonumber
\end{eqnarray}
where ${\rm Op}^w_\hbar (H)$ is the Weyl quantization of $H$ as in  (\ref{weyl}). As for the initial datum,  we can  require  $\varphi_{\hbar} \in W^{2,2} (\Bbb T^n;\Bbb C)$ and $\|  \varphi_\hbar \|_{L^2} \le C$ $\forall$ $0 < \hbar \le 1$. The one parameter group of unitary operators $e^{-\frac{i}{\hbar}  \mathrm{Op}^w_{\hbar} (H)  t}$ can be defined on the whole $L^{2}  (\mathbb{T}^n; \Bbb C)$.
In fact,  this is because the Schr\"odinger operator $\hat{H}_\hbar := - \frac{1}{2}\hbar^2 \Delta_x + V(x)$ is coinciding with ${\rm Op}^w_\hbar (H)$. This is the content of the Lemma \ref{equi-op} shown in the Appendix.

\subsection{The equation for the Wigner transform}
\noindent
In this section we provide a result on the equation for the Wigner transform of the solution of the Schr\"odinger equation written on the torus. The well known arguments  within the framework of the Weyl quantization on $\Bbb R^n$ (see  \cite{A-F-G}, \cite{A-F-P}, \cite{L-P}) must  be adapted for the Weyl quantization on $\Bbb T^n$.\\

\noindent
The first result reads as follows
\begin{proposition}
\label{TH21}
Let $\psi_\hbar$ be the solution of  (\ref{eqSch}), and $f \in C^\infty ([0,t] \times \mathbb{T}^n \times \Bbb R^n;\Bbb R)$ such that $\forall s \in [0,t]$ it holds $f(s,\cdot) \in A$  as in Def \ref{def-wc}. Then, 
\begin{equation}
\label{w-trans}
\int_0^t \sum_{ \eta \in \frac{\hbar}{2} \mathbb{Z}^n}\int_{\mathbb{T}^n}  \Big[ \Big( \partial_s f    +   \eta \cdot  \nabla_x  f   \Big) (s,x,\eta) W_\hbar \psi_\hbar (s,x,\eta) +   f (s,x,\eta)    \mathcal{E}_\hbar \psi_\hbar (s,x,\eta) \Big] dxds = 0
\\
\end{equation}
where 
\begin{eqnarray}
\label{E0}
\mathcal{E}_\hbar \psi_\hbar (s,x,\eta) := \frac{i}{(2\pi)^n\hbar}   \int_{\Bbb{T}^n} \displaystyle{e^{ 2 \frac{i}{\hbar} \langle z,\eta \rangle} }  \{V(x+z) - V(x-z) \}  \psi_\hbar (s,x-z)\overline{\psi}_\hbar (s,x+z)  dz. 
\end{eqnarray}
\end{proposition}
\begin{proof}
We interpret all the subsequent partial derivatives in the distributional sense of $A^\prime$.  
To begin,
\begin{eqnarray}
\partial_t W_\hbar \psi  &=&  (2\pi)^{-n} \int_{\Bbb{T}^n} \displaystyle{e^{ 2 \frac{i}{\hbar} \langle z,\eta \rangle} }  \partial_t  \psi_\hbar (t,x-z)\overline{\psi}_\hbar (t,x+z)  dz
\nonumber\\
&+&  (2\pi)^{-n} \int_{\Bbb{T}^n} \displaystyle{e^{ 2 \frac{i}{\hbar} \langle z,\eta \rangle} }  \psi_\hbar (t,x-z) \partial_t  \overline{\psi}_\hbar  (t,x+z) dz.
\label{partialW}
\end{eqnarray}
Since $\psi_\hbar$ solves the Schr\"odinger equation, it follows
\begin{eqnarray}
&& \partial_t  \psi_\hbar (t,x-z)\overline{\psi} (t,x+z) + \psi_\hbar (t,x-z) \partial_t  \overline{\psi}_\hbar  (t,x+z)
\\
&=&  \frac{i\hbar}{2} [  (\Delta_x  \psi_\hbar (t,x-z) ) \overline{\psi}_\hbar  (t,x+z)  -  \psi_\hbar (t,x-z)  \Delta_x \overline{\psi}_\hbar  (t,x+z) ]
\label{g43}\\
&+&  \ i \hbar^{-1}  [   V(x+z) - V(x-z)  ] \psi_\hbar (t,x-z)\overline{\psi}_\hbar (t,x+z).
\nonumber
\end{eqnarray} 
Now recall the simple equality $(\Delta_x f)  g -  f  \Delta_x g = {\rm div}_x [(\nabla_x f)  g -  f  \nabla_x g]$, so that 
\begin{eqnarray}
&& (\Delta_x  \psi_\hbar (t,x-z) ) \overline{\psi}_\hbar  (t,x+z)  -  \psi_\hbar (t,x-z)  \Delta_x \overline{\psi}_\hbar  (t,x+z) 
\nonumber\\
&=& 2 \, {\rm div}_x \nabla_z  [  \psi_\hbar (t,x-z)  \overline{\psi}_\hbar  (t,x+z)   ].
\label{s-equi}
\end{eqnarray}
Then, insert (\ref{s-equi}) in (\ref{g43}), so that 
\begin{eqnarray}
&& \partial_t  \psi_\hbar (t,x-z)\overline{\psi} (t,x+z) + \psi_\hbar (t,x-z) \partial_t  \overline{\psi}_\hbar  (t,x+z)
\\
&=&  i \hbar   \  {\rm div}_x \nabla_z  [  \psi_\hbar (t,x-z)  \overline{\psi}_\hbar  (t,x+z)   ] 
 + \frac{ i}{ \hbar} \ [   V(x+z) - V(x-z)  ] \psi_\hbar (t,x-z)\overline{\psi}_\hbar (t,x+z).
\label{W-t}
\end{eqnarray} 
Moreover, an easy computation involving integration by parts shows
\begin{equation}
\eta \cdot \nabla_x W_\hbar \psi_\hbar =  - i \hbar \ (2\pi)^{-n}\int_{\Bbb{T}^n} e^{2 \frac{i}{\hbar} \langle z,\eta\rangle} {\rm div}_z \nabla_x  [  \psi_\hbar (t,x-z)  \overline{\psi}  (t,x+z)   ]dz .
\label{W-eta}
\end{equation}
Hence, by  (\ref{W-t}) and (\ref{W-eta}) we directly get the statement.
\end{proof}

\begin{lemma}
\label{reg-delta}
Let $\epsilon >0$ and $g(\epsilon,\cdot \,): \Bbb T^n \longrightarrow \Bbb R^+$  defined as
\begin{equation}
g (\epsilon,y)  := \frac{1}{(2\pi)^{n}} \sum_{\kappa_0 \in \Bbb Z^n}  e^{-\epsilon |\kappa_0|^2}  \displaystyle{e^{ - i \langle y , \kappa_0 \rangle} } = \frac{1}{(2\pi)^{n}} \sum_{\xi \in \Bbb Z^n}  \Big(  \frac{\pi}{ \epsilon} \Big)^{n\over 2}  e^{-   |\xi - y|^2  (4\epsilon)^{-1}}.
\end{equation}
Then, $\forall \psi \in C^\infty (\Bbb T^n;\Bbb C)$
\begin{equation}
\lim_{\epsilon \rightarrow 0^+} \int_{\Bbb T^n} g (\epsilon,y - y_0)  \psi(y_0)  dy_0 = \psi(y).
\end{equation}
\end{lemma}
\begin{proof}
Let $G (\kappa_0,\epsilon,y) := e^{-\epsilon |\kappa_0|^2}  e^{ - i \langle y , \kappa_0 \rangle} $, then $\widehat{G} (\xi,\epsilon,y) := \int_{\Bbb R^n} e^{-i  \langle \xi , \kappa_0 \rangle} G (\kappa_0,\epsilon,y) d\kappa_0$ reads 
$$ 
\widehat{G} (\xi,\epsilon,y) =  \Big(  \frac{\pi}{ \epsilon} \Big)^{n\over 2} e^{-    | \xi - y|^2  (4\epsilon)^{-1}}
$$
By applying the Poisson's summation formula (see for example \cite{D-K}), 
\begin{equation}
\label{ge2}
g (\epsilon,y)  = \frac{1}{(2\pi)^{n}} \sum_{\xi \in \Bbb Z^n}  \Big(  \frac{\pi}{ \epsilon} \Big)^{n\over 2}  e^{-   |\xi - y|^2  (4\epsilon)^{-1}} = \frac{1}{(2\pi)^{n}} \sum_{\xi \in \Bbb Z^n}  \Big(  \frac{\pi}{ \epsilon} \Big)^{n\over 2}  e^{-   |2\pi \xi - 2\pi y|^2  (16 \pi^2 \epsilon)^{-1}}    .
\end{equation}
Now recall the identification $\Bbb T^n = (\Bbb R / 2\pi \Bbb Z)^n$, fix the periodicity domain $y_0 \in Q_n := [0,2\pi]^n$, so  that
\begin{eqnarray}
\lefteqn{ \lim_{\epsilon \rightarrow 0^+} \int_{Q_n}  g (\epsilon,y - y_0 )  \psi(y_0 )  dy_0 }
\\
&=&\lim_{\epsilon \rightarrow 0^+} \Big(  \frac{1}{ 4\pi \epsilon} \Big)^{n\over 2}    \sum_{\xi \in \Bbb Z^n} \int_{\Bbb R^n}  e^{-   |2\pi \xi - 2\pi (y - y_0)|^2  (16 \pi^2 \epsilon)^{-1}}   \psi(y_0)  \mathcal{X}_{Q_n} (y_0) dy_0 
\\
&=&\lim_{\epsilon \rightarrow 0^+} \Big(  \frac{1}{ 4\pi \epsilon} \Big)^{n\over 2}    \int_{\Bbb R^n}  e^{-   |y - y_0 |^2  (4  \epsilon)^{-1}}   \psi(y_0)  \mathcal{X}_{Q_n} (y_0) dy_0
=  \psi(y).
\end{eqnarray} 
\end{proof}

In the following, we provide the evolution equation for the Wigner transform $W_\hbar \psi_\hbar$ of the solution of the Schr\"odinger's equation on the torus, 
\begin{eqnarray}
\label{eq-transW}
\partial_t W_\hbar \psi_\hbar + \eta \cdot \nabla_x W_\hbar \psi_\hbar + \mathcal{K}_\hbar \star_\eta W_\hbar \psi_\hbar = 0
\end{eqnarray}
written in the distributional sense. More precisely, $\forall f \in C^\infty ([0,t] \times \mathbb{T}^n \times \Bbb R^n;\Bbb R)$ such that  $f(s,\cdot) \in A$ $\forall s \in [0,t]$ as in Def \ref{def-wc} it holds
\begin{equation}
\label{w-trans01}
\int_0^t \sum_{ \eta \in \frac{\hbar}{2} \mathbb{Z}^n}\int_{\mathbb{T}^n}  \Big[ \Big( \partial_s f    +   \eta \cdot  \nabla_x  f   \Big) (s,x,\eta) W_\hbar \psi_\hbar (s,x,\eta) +   f  (s,x,\eta) \  \mathcal{K}_\hbar \star_\eta  W_\hbar \psi_\hbar   (s,x,\eta) \Big] dxds = 0
\\
\end{equation}
where for $ \eta \in \frac{\hbar}{2} \Bbb Z^n$
\begin{eqnarray}
\label{Kappa}
 \mathcal{K}_\hbar (s,x,\eta) &:=& \frac{i}{(2\pi)^n\hbar}   \int_{\Bbb{T}^n} \displaystyle{e^{ 2 \frac{i}{\hbar} \langle z,\eta \rangle} }  \{V(x+z) - V(x-z) \}   dz, 
 \\
 \mathcal{K}_\hbar \star_\eta  W_\hbar \psi_\hbar   (s,x,\eta) &:=& \sum_{\kappa_0 \in \Bbb Z^n}  \mathcal{K}_\hbar \Big(s,x, \eta - \frac{\hbar}{2} \kappa_0 \Big)  W_\hbar \psi_\hbar   \Big(s,x,\frac{\hbar}{2} \kappa_0 \Big) .
 \label{Kappa2}
\end{eqnarray}

\begin{theorem}
\label{TH22}
Let $\psi_\hbar$ be the solution of  (\ref{eqSch}). Then, it holds 
\begin{equation}
\label{w-trans}
\partial_t W_\hbar \psi_\hbar + \eta \cdot \nabla_x W_\hbar \psi_\hbar + \mathcal{K}_\hbar \star_\eta W_\hbar \psi_\hbar = 0
\end{equation}
in the distributional sense as in (\ref{w-trans01}).
\begin{proof}
We exhibit a short proof based on the previous result, namely we simply show that convolution  (\ref{Kappa2}) is well defined and coincides with the remainder term (\ref{E0}). Since $V \in C^\infty (\Bbb T^n;\Bbb R)$, the related Fourier components $V_\omega := (2\pi)^{-n} \int_{\Bbb T^n}  e^{i \omega z} V(z) dz$, $\omega \in \Bbb Z^n$, fulfill $|V_\omega| \le c_j \langle \omega \rangle^{j}$ $\forall j \in \Bbb N$ and some $c_j >0$. An easy computation shows that
\begin{equation}
 \mathcal{K}_\hbar \Big(s,x, \frac{\hbar}{2} \kappa \Big) =   \frac{i}{(2\pi)^n\hbar}  (e^{- i \kappa \cdot x} V_\kappa -  e^{+ i \kappa \cdot x} V_\kappa^\star ), \quad \kappa \in \Bbb Z^n.
\end{equation}
Moreover, $\| W_\hbar \psi_\hbar   (s,\cdot)  \|_\infty \le (2\pi)^{-n} C^2$ $\forall s \in \Bbb R$. Thus, the series in   (\ref{Kappa2})  is absolutely convergent, and we can write down the regularization (useful in the subsequent computations): 
\begin{eqnarray}
\mathcal{K}_\hbar \star_\eta  W_\hbar \psi_\hbar  &=& \lim_{\epsilon \rightarrow 0^+} \sum_{\kappa_0 \in \Bbb Z^n} e^{-\epsilon |\kappa_0|^2} \mathcal{K}_\hbar \Big(s,x, \eta - \frac{\hbar}{2} \kappa_0 \Big)  W_\hbar \psi_\hbar   \Big(s,x,\frac{\hbar}{2} \kappa_0 \Big) .
\end{eqnarray}
We look at the regularization:
\begin{eqnarray}
\label{247-k}
\lefteqn{ \sum_{\kappa_0 \in \Bbb Z^n} e^{-\epsilon |\kappa_0|^2}  \mathcal{K}_\hbar \Big(s,x, \eta - \frac{\hbar}{2} \kappa_0 \Big)  W_\hbar \psi_\hbar   \Big(s,x,\frac{\hbar}{2} \kappa_0 \Big) }
\\
&=& \sum_{\kappa_0 \in \Bbb Z^n}   e^{-\epsilon |\kappa_0|^2}  \frac{i}{(2\pi)^{n}\hbar}    \int_{\Bbb{T}^n} \displaystyle{e^{ 2 \frac{i}{\hbar} \langle z,\eta - \frac{\hbar}{2} \kappa_0 \rangle} }  \{V(x+z) - V(x-z) \}   dz 
\\
&\times&
\frac{1}{(2\pi)^{n}}   \int_{\Bbb{T}^n} e^{2 \frac{i}{\hbar} \langle \tilde{z},\frac{\hbar}{2} \kappa_0 \rangle}\psi_\hbar (s,x-\tilde{z}) \psi_\hbar^\star (s,x+\tilde{z}) d \tilde{z}
\nonumber
\\
&=& \frac{i}{(2\pi)^{n}\hbar} \int_{\Bbb T^n} \int_{\Bbb{T}^n}  \displaystyle{e^{ 2 \frac{i}{\hbar} \langle z,\eta  \rangle} }  \, \Big[ \frac{1}{(2\pi)^{n}} \sum_{\kappa_0 \in \Bbb Z^n}  e^{-\epsilon |\kappa_0|^2}  \displaystyle{e^{ - i \langle z - \tilde{z} , \kappa_0 \rangle} }   \Big]
\\
&\times& \{V(x+z) - V(x-z) \}   \psi_\hbar (s,x-\tilde{z}) \psi_\hbar^\star (s,x+\tilde{z}) dz d \tilde{z}   \, 
\nonumber
\end{eqnarray}
However, for any fixed $\epsilon > 0$, the function
\begin{equation}
g (\epsilon, z- \tilde{z})  := \frac{1}{(2\pi)^{n}} \sum_{\kappa_0 \in \Bbb Z^n}  e^{-\epsilon |\kappa_0|^2}  \displaystyle{e^{ - i \langle z - \tilde{z} , \kappa_0 \rangle} }  
\end{equation}
defines a tempered distribution on $C^\infty (\Bbb T^n;\Bbb C)$ converging to $\delta(z- \tilde{z})$ as $\epsilon \rightarrow 0^+$ (see Lemma \ref{reg-delta}).\\ 
To conclude,
\begin{eqnarray}
\nonumber
\lefteqn{\mathcal{K}_\hbar \star_\eta  W_\hbar \psi_\hbar  }
\\
&=&  \lim_{\epsilon \rightarrow 0^+}  \frac{i}{(2\pi)^n\hbar} \int_{\Bbb T^n} \int_{\Bbb{T}^n}  \displaystyle{e^{ 2 \frac{i}{\hbar} \langle z,\eta  \rangle} } 
g (\epsilon, z- \tilde{z})  \{V(x+z) - V(x-z) \}   \psi_\hbar (s,x-\tilde{z}) \psi_\hbar^\star (s,x+\tilde{z}) dz  d \tilde{z}  \, 
\nonumber
\\
&=&  \frac{i}{(2\pi)^n\hbar} \int_{\Bbb T^n}  \lim_{\epsilon \rightarrow 0^+} \int_{\Bbb{T}^n}  \displaystyle{e^{ 2 \frac{i}{\hbar} \langle z,\eta  \rangle} } 
g (\epsilon, z- \tilde{z})  \{V(x+z) - V(x-z) \}   \psi_\hbar (s,x-\tilde{z}) \psi_\hbar^\star (s,x+\tilde{z}) dz d \tilde{z}   \, 
\nonumber
\\
&=&  \frac{i}{(2\pi)^n\hbar}  \int_{\Bbb{T}^n}  \displaystyle{e^{ 2 \frac{i}{\hbar} \langle \tilde{z} ,\eta  \rangle} }  \{V(x+\tilde{z}) - V(x-\tilde{z}) \}   \psi_\hbar (s,x-\tilde{z}) \psi_\hbar^\star (s,x+\tilde{z}) d\tilde{z}
=: \mathcal{E}_\hbar \psi_\hbar. 
\nonumber
\end{eqnarray}
\end{proof}

\end{theorem}

\section{Semiclassical limits of Wigner transforms on the torus}\label{scwt}    \markboth{Semiclassical limits of Wigner transforms on the torus}{}

\subsection{The Liouville equation}

This section is devoted to the Liouville equation written in the measure sense on $\Bbb T^n \times \Bbb R^n$ solved  by the semiclassical asymptotics of the toroidal Wigner transform. 

\begin{theorem}
\label{TH41}
Let $\psi_\hbar(t) := e^{-\frac{i}{\hbar}  \mathrm{Op}^w_{\hbar} (H)  t} \varphi_\hbar$ where $\varphi_\hbar \in L^2 (\Bbb T^n;\Bbb C)$ and $\|  \varphi_\hbar \|_{L^2} \le C$. Let   $\{w_t \}_{t \in [-T,T]}$  be a limit of  $W_\hbar \psi_\hbar (t)$ in $L^\infty ([-T,+T];A^\prime)$ along a sequence of values of $\hbar\to 0$. Then,  
\begin{equation}
\label{eq-Li1}
\partial_t w_t + \eta \cdot \nabla_x w_t  - \nabla_x V(x) \cdot \nabla_\eta  w_t = 0
\end{equation}
in the distributional sense.
\end{theorem}
\begin{proof}
To begin, we prove that
\begin{equation}
\label{eq-Liou}
\frac{d}{dt}  \int_{\mathbb{T}^n \times \Bbb R^n}  \phi(x,\eta) dw_t (x,\eta)  +   \int_{\mathbb{T}^n \times \Bbb R^n}   \{ \phi, H \} (x,\eta)  dw_t (x,\eta)   = 0
\end{equation}
for any $\phi \in A$, see (\ref{setA2}).  To this aim, we observe that the Schr\"odinger equation implies 
\begin{equation}
\label{eq-Hg}
\frac{d}{dt} \langle \psi_\hbar (t) , \mathrm{Op}^w_{\hbar} (\phi)  \psi_\hbar (t) \rangle_{L^2} =  (i \hbar)^{-1}  \langle \psi_\hbar (t) ,  [ \mathrm{Op}^w_{\hbar} (H),   \mathrm{Op}^w_{\hbar} (\phi) ]  \psi_\hbar (t) \rangle_{L^2}.
\end{equation}
%This is a well defined equivalence since $\psi_\hbar (t,\cdot) \in H^{2}(\Bbb T^n;\Bbb C)$ (see Remark \ref{L-reg}), $\mathrm{Op}^w_{\hbar} (H)  = \hat{H}_\hbar$ on $C^\infty (\Bbb T^n;\Bbb C)$ (see Lemma \ref{equi-op}) therefore defined also on $H^{2}(\Bbb T^n;\Bbb C)$, and $\mathrm{Op}^w_{\hbar} (\phi)$ can be extended to an $L^2(\Bbb T^n)$-bounded operator since $\phi \in  S^{m} (\Bbb T^n \times \Bbb R^n)$ for any $m \in \Bbb R$ (see Theorem  \ref{Th-Bound0}). 
Hence
\begin{equation}
\label{eq-Hg22}
 \langle \psi_\hbar (t) , \mathrm{Op}^w_{\hbar} (\phi)  \psi_\hbar (t) \rangle_{L^2}  -  \langle \varphi_\hbar , \mathrm{Op}^w_{\hbar} (\phi)  \varphi_\hbar  \rangle_{L^2}  = \int_0^t (i \hbar)^{-1}  \langle \psi_\hbar (s) ,  [ \mathrm{Op}^w_{\hbar} (H),   \mathrm{Op}^w_{\hbar} (\phi) ]  \psi_\hbar (s) \rangle_{L^2} ds.
\end{equation}
where $\psi_\hbar (t=0) =:  \varphi_\hbar \in L^{2} (\Bbb T^n;\Bbb C)$ with $\| \varphi_\hbar\|_{L^2} \le C$ $\forall 0 < \hbar \le 1$.
Morever, thanks to Theorem \ref{Th-comp}, the Weyl simbol of the commutator (namely the Moyal bracket of simbols $H$ and $\phi$) reads
\begin{equation}
\{ H, \phi \}_{{\rm M}}   =   i \hbar \{ H, \phi \} + r 
\end{equation}
where  $r$ has order $O(\hbar^2)$ when estimated in $S^{2+m} (\Bbb T^n \times \Bbb R^n)$ for any $m \in \Bbb R$, and thus also in $S^{0} (\Bbb T^n \times \Bbb R^n)$, 
\begin{equation}
|  \partial_x^\beta \partial_\eta^\alpha  r (x,\eta) |   \le  C_{\alpha \beta} \, \hbar^2 \langle \eta \rangle^{-  |\alpha| }.
\end{equation}
The related remainder operator $\mathrm{Op}^w_{\hbar} (r)$ is thus $L^2$-bounded, with (time independent) norm estimate thanks to Theorem \ref{Th-Bound0} with order $O(\hbar^2)$. This directly gives 
\begin{equation}
\lim_{\hbar \rightarrow 0^+} \hbar^{-1} \Big| \int_0^t \langle \psi_\hbar (s) ,  \mathrm{Op}^w_{\hbar} (r)  \psi_\hbar (s)  \rangle_{L^2}  ds \Big| \le \lim_{\hbar \rightarrow 0^+} t \,  \hbar^{-1} \|  \mathrm{Op}^w_{\hbar} (r) \|_{L^2 \rightarrow L^2} = 0,
\end{equation}
since   $\| \psi_\hbar (s)\|_{L^2} = \| \psi_\hbar (s=0)\|_{L^2} = \| \varphi_\hbar\|_{L^2} \le C$.
The first term in (\ref{eq-Hg22}) reads
\begin{equation}
\label{348}
 \langle \psi_\hbar (t) , \mathrm{Op}^w_{\hbar} (\phi)  \psi_\hbar (t) \rangle_{L^2} =   \sum_{\eta \in \frac{\hbar}{2} \mathbb{Z}^n} \int_{\mathbb{T}^n} \phi (x,\eta) W_\hbar \psi_\hbar (t,x,\eta)dx.
\end{equation}
Let  $w_t (x,\eta) $ be  a family of Radon measures of finite mass  on $\Bbb T^n \times \Bbb R^n$  for any $t \in [-T,T]$ which is a limit of $W_\hbar \psi_\hbar$ in 
$L^\infty ([-T,+T];A^\prime)$ along  a sequence of values of $\hbar\to 0$ . The related semiclassical limit of (\ref{348}) reads 
\begin{equation}
\int_{\mathbb{T}^n \times \Bbb R^n}  \phi(x,\eta) dw_t (x,\eta).
\end{equation}
If we now look at 
\begin{equation}
\sum_{\eta \in \frac{\hbar}{2} \mathbb{Z}^n} \int_{\mathbb{T}^n} \{ H, \phi \} (x,\eta) W_\hbar \psi_\hbar (t,x,\eta)dx
\end{equation}
we recall that $\phi$ is rapidly decreasing in $\eta$-variables and the phase space transform $\widehat{\phi}$ has compact support, hence also $\{ H, \phi \} \in A$. As a consequence, we can extract a subsequence of the above one so  that the  semiclassical limit of the righthand side of (\ref{eq-Hg22}) reads
\begin{equation}
\int_0^t \int_{\mathbb{T}^n \times \Bbb R^n}   \{ H, \phi \} (x,\eta)  dw_s (x,\eta) ds.  
\end{equation}
We therefore deduce that $\forall  t \in \Bbb R$
\begin{equation}
 \int_{\mathbb{T}^n \times \Bbb R^n}  \phi(x,\eta) dw_t (x,\eta) - \int_{\mathbb{T}^n \times \Bbb R^n}  \phi(x,\eta) dw_0 (x,\eta) =  \int_0^t \int_{\mathbb{T}^n \times \Bbb R^n}   \{ H, \phi \} (x,\eta)  dw_s (x,\eta) ds,
\end{equation}
and observe that the righthand side is differentiable for any $t \in \Bbb R$ (and thanks to the equivalence, the lefthand side too). 
We now take the time derivative of both sides and get equation (\ref{eq-Liou}). On the other hand, since $H$ is smooth,  it is easily seen  that equation (\ref{eq-Liou}) has a unique solution, and it is given by the push forward of the initial data  $w_t = (\varphi_H^t)_\star (w_0)$ involving the Hamiltonian flow. However, this is also the unique solution of the Liouville equation written in the following weak sense 
\begin{equation}
\label{eq-Li22}
\int_0^t \int_{\mathbb{T}^n \times \Bbb R^n} [ \partial_t f(t,x,\eta) +   \{ f, H \} (t,x,\eta) ] \ dw_s (x,\eta) ds  = 0 \quad \forall f \in C^\infty_0 ([0,t] \times \Bbb T^n \times \Bbb R^n), 
\end{equation}
see for example \cite{A}. 
\end{proof}

\subsection{WKB wave functions of positive and negative type}
\label{SEC-wkb}

We begin this section introducing a class of WKB-type wave functions in $H^{1}(\Bbb T^n;\Bbb C)$ associated with weak KAM solutions of the stationary Hamilton-Jacobi equation.

\begin{definition}
\label{wave-d}
Let $P \in \ell \, \Bbb Z^n$ for some $\ell >0$ and   $\hbar^{-1} \in  \ell^{-1} \Bbb N$.  
Let $v_{\pm} (P,\cdot) \in C^{0,1} (\Bbb T^n;\Bbb R)$ be  weak KAM solutions of the H-J equation (\ref{def-eff0}) (in the sense of \cite{F}, see  subsection \ref{sub-wHJ}). Select $a_{\hbar,P}^\pm   \in H^{1}(\Bbb T^n; \Bbb R^+)$ such that
\begin{equation}
\label{def-inc}
{\rm dom}(a_{\hbar,P}^\pm ) \subseteq  {\rm dom}(\nabla_x v_\pm (P ,\cdot))
\end{equation}
$\|  a_{\hbar,P}^\pm   \|_{L^2} = 1$ and $\hbar \, \|  a_{\hbar,P}^\pm   \|_{H^{1}} \longrightarrow 0$ as $\hbar \longrightarrow 0^+$. We suppose that the following weak limit upon passing through a subsequence $\hbar_j \longrightarrow 0^+$
\begin{equation}
\label{def-dm}
\exists \ dm_P^\pm (x) :=  \lim_{\hbar_j \longrightarrow 0^+} \ |a_{\hbar_j,P}^\pm (x)|^2 dx 
\end{equation} 
fulfills $dm_P \ll \pi_\star(dw_P) =: d\sigma_P$ where $dw_P$ is a Mather $P$-minimal measure as in (\ref{def-dwP}). The WKB wave functions of negative type are defined by
\begin{equation}
\label{wf-d}
\varphi_{\hbar}^- (x) := a_{\hbar,P}^- (x) \ e^{\frac{i}{\hbar} [P \cdot x + v_- (P,x)]  }.
\end{equation}  
The WKB wave functions of positive type are given by
\begin{equation}
\label{wf-dp}
\varphi_{\hbar}^+ (x) := a_{\hbar,P}^+ (x) \ e^{\frac{i}{\hbar} [P \cdot x + v_+ (P,x)]  }.
\end{equation}   
\end{definition}

\begin{remark}[{\bf Example}]
\label{EX-a}
About the previous definition, we exhibit an explicit construction  for $a_{\hbar,P}^\pm$.    
In fact, consider  $\rho \in C_0^\infty(\mathbb{R}^n)$ such that $0\leq \rho$, $\mathrm{supp}\,\rho \subset Q_n := [0,2\pi]^n$ 
and $\int\rho (x) dx=1$. For a fixed $\alpha>0$ let
\begin{equation}
\Phi_{\alpha,\hbar }(x):= \hbar^{-n\alpha}\sum_{k\in\mathbb{Z}^n}\rho \Bigl(\frac{x-2\pi k}{\hbar^\alpha}\Bigr).
\label{eqMillifier}\end{equation}
Then $\displaystyle{\int_{\mathbb{T}^n}}\Phi_{\alpha,\hbar}(x)dx=1,$ and if $f\in L^1(\mathbb{T}^n)$ we have, by the periodicity,
$$\Phi_{\alpha,\hbar}\star f(x)=\int_{\mathbb{T}^n}\Phi_{\alpha,\hbar}(x-y)f(y)dy=\int_{Q_n} \rho(z)f(x-\hbar^\alpha z)dz$$
Fix a fixed ($P$-dependent)  Borel positive measure $d m_P^\pm$ on $\Bbb T^n$ with ${\rm supp} (d m_P^\pm) \subseteq  {\rm dom}(\nabla_x v_\pm (P ,\cdot))$,  an amplitude function can be given by
\begin{equation}
\label{ampl}
a_{\hbar,P}^\pm(x)  :=  \Big\{  \int_{\mathbb{T}^n} \frac{1}{c_0} \Big( \hbar^\epsilon +  \Phi_{\gamma,\hbar}(x-y)\Big) d m_P (y)  \Big\}^{1/2} \Big|_{ {\rm dom}( \nabla v_\pm)}, 
\end{equation}
where $ \epsilon,  \gamma > 0$ with  $0 < \epsilon +  \gamma(n+1) < 1$,  $c_0=c_0(\hbar)=|\!| \hbar^\epsilon+\rho |\!|_{L^1(Q_n)}=
1+O(\hbar^\epsilon)$.  Notice that $a>\hbar^{\epsilon/2} c_0^{-1/2}$ then since $x \mapsto a_{\hbar,P}^\pm(x)$ is $2\pi$-periodic (in each variable), it is a well-defined function on the torus. 
The function (\ref{ampl}) is in $C^k (\Bbb T^n ; \Bbb R^+)$, $\forall$ $k \in \Bbb N$,  and fulfills (see Prop 4.5 in \cite{P-Z})
\begin{itemize}
\item[{\bf (i)}]
$
\displaystyle{\int_{\mathbb{T}^n}    |a_{\hbar,P}^\pm(x)|^2   dx  =  1}; 
$
\item[{\bf (ii)}]
$
\displaystyle{
\hbar^2  \int_{\mathbb{T}^n}   | \nabla_x a_{\hbar,P}^\pm (x)|^2   dx  \le     |\!|\nabla_x \rho  |\!|^2_{L^\infty} \, \hbar^{2 (1 - \epsilon - (n+1)\gamma)  } };
$
\item[{\bf (iii)}]
$
\displaystyle{
\lim_{\hbar \rightarrow 0+} \int_{\mathbb{T}^n}  f(x)  |a_{\hbar,P}^\pm (x)|^2 dx = \int_{\mathbb{T}^n} f(x) dm_P^\pm (x), \quad \forall f  \in  
C^{0}(\mathbb{T}^n;\mathbb{R}),
}
$
\item[{\bf (iv)}] 
$\displaystyle{\lim_{\hbar \rightarrow 0+} \int_{\mathbb{T}^n}  f(x)  |a_{\hbar,P}^\pm (x)|^2 dx = \int_{\mathbb{T}^n} f(x) dm_P^\pm (x),}$ 
$\forall$ bounded Borel measurable $f \colon\mathbb{T}^n\longrightarrow\mathbb{R}$ whose discontinuity set has zero $dm_P$-measure. 
\end{itemize}
\end{remark}

In the following, we provide two useful Lemma involving our class of WKB functions. 

\begin{lemma}
Let $\varphi_{\hbar}^\pm$ be as in Definition \ref{wave-d}. Then,  $\varphi_{\hbar}^\pm \in H^1(\Bbb T^n;\Bbb C)$. 
\end{lemma}
\begin{proof}
The $L^2$-norm simply reads
$
\| \varphi_{\hbar}^\pm  \|_{L^2} = \|   a_{\hbar,P}^\pm  \|_{L^2} < + \infty,
$
whereas
$$
\| \nabla_x \varphi_{\hbar}^\pm  \|_{L^2} \le \frac{1}{\hbar} \| (P + \nabla_x v_\pm ) a_{\hbar,P}^\pm    \|_{L^2} +  \| \nabla_x a_{\hbar,P}^\pm \|_{L^2}
$$
Recalling (\ref{inc-G}) and the setting of $a_{\hbar,P}$, it follows 
$$
\| \nabla_x \varphi_{\hbar}^\pm  \|_{L^2} \le \frac{1}{\hbar} \| P + \nabla_x v_\pm    \|_{L^\infty} +  \| a_{\hbar,P}^\pm \|_{H^1} < + \infty \quad \forall \ 0 < \hbar \le 1.
$$
\end{proof}

\begin{lemma}
\label{L-T1}
Let $\varphi_{\hbar}^\pm$ be as in Definition \ref{wave-d}. Let $\Bbb P_\hbar^{\pm}$ be as in (\ref{PO}) associated to  $\varphi_{\hbar}^\pm$. Then, the family of measures $\{ \Bbb P_\hbar^{\pm} \}_{0 \le \hbar \le 1}$ is tight.
\end{lemma}
\begin{proof}
Let  $M_R := \Bbb T^n \times \{ \Bbb R^n \backslash B_R  \}$ and $U_R := \Bbb R^n \backslash B_R $. Thanks to (\ref{P-U}) 
\begin{eqnarray}
\label{s-UR}
\mathbb{P}^\pm_\hbar  (\Bbb T^n \times U_R)  =  \sum_{ \alpha \in \mathbb{Z}^n}  \mathcal{X}_{U_R}  (\hbar \alpha)  |\widehat{\phi}_{\hbar,\alpha}^\pm|^2 
\end{eqnarray}
where the Fourier components read
\begin{eqnarray}
\widehat{\phi}_{\hbar,\alpha}^\pm &:=& (2\pi)^{-n} \int_{\Bbb T^n}  e^{- i \alpha \cdot x}  \varphi_{\hbar}^\pm (x) dx =  
(2\pi)^{-n} \int_{\Bbb T^n}  e^{- i \alpha \cdot x}  a_{\hbar,P}^\pm (x) \ e^{\frac{i}{\hbar} [P \cdot x + v_\pm (P,x)]  } dx
\\
&=& (2\pi)^{-n} \int_{\Bbb T^n}  a_{\hbar,P}^\pm (x) \ e^{\frac{i}{\hbar}  v_\pm (P,x)  }  e^{ \frac{i}{\hbar} (- \hbar \alpha + P) \cdot x}   dx
\end{eqnarray} 
and $P \in \ell \Bbb Z^n$ for some fixed $\ell >0$; moreover we underline thar the series (\ref{s-UR}) is computed over $|\hbar \alpha| > R$ (or equivalently $|\alpha| > R \hbar^{-1}$). In the case $R > |P|$, it holds the equality
\begin{eqnarray}
\widehat{\phi}_{\hbar,\alpha}^\pm &=&\frac{(-i\hbar) }{|- \hbar \alpha + P|^2} (- \hbar \alpha + P) \cdot (2\pi)^{-n} \int_{\Bbb T^n}    a_{\hbar,P}^\pm (x) \ e^{\frac{i}{\hbar}  v_\pm (P,x)  } \nabla_x  e^{ \frac{i}{\hbar} (- \hbar \alpha + P) \cdot x} dx.
\end{eqnarray} 
The integration by parts gives
\begin{eqnarray}
\widehat{\phi}_{\hbar,\alpha}^\pm &=&\frac{(i\hbar) }{|- \hbar \alpha + P|^2} (- \hbar \alpha + P) \cdot (2\pi)^{-n} \int_{\Bbb T^n}  \nabla_x   a_{\hbar,P}^\pm (x) \ e^{\frac{i}{\hbar}  v_\pm (P,x)  }  e^{ \frac{i}{\hbar} (- \hbar \alpha + P) \cdot x} dx
\\
&-&\frac{1 }{|- \hbar \alpha + P|^2} (- \hbar \alpha + P) \cdot (2\pi)^{-n} \int_{\Bbb T^n}   a_{\hbar,P}^\pm (x) ( \nabla_x v_\pm (P,x) ) e^{\frac{i}{\hbar}  v_\pm (P,x)  }  e^{ \frac{i}{\hbar} (- \hbar \alpha + P) \cdot x} dx
\nonumber
\end{eqnarray}
We are now in the position to provide an estimate for  $|\widehat{\phi}_{\hbar,\alpha}^\pm|$, indeed some easy computations together with the application of Cauchy-Schwarz inequality give
\begin{eqnarray}
|\widehat{\phi}_{\hbar,\alpha}^\pm | &\le&  \frac{ (2\pi)^{-n/2} }{|- \hbar \alpha + P|} \Big( \| \hbar \nabla_x   a_{\hbar,P}^\pm \|_{L^2} + \| \nabla_x v_\pm (P, \cdot \,) \|_{L^\infty}   \Big)
\end{eqnarray}
Recalling (\ref{inc-G}) we have  $ \| \nabla_x v_\pm (P,\cdot \,) \|_{L^\infty} < + \infty$ for any fixed $P \in \ell \Bbb Z^n$. We also remind that $\| \hbar \nabla_x   a_{\hbar,P}^\pm \|_{L^2}  \rightarrow 0 $ as $\hbar \rightarrow 0^+$.
To conclude, by defining 
\begin{equation}
C_{n,P} := (2\pi)^{-n} \Big( \sup_{0< \hbar \le 1} (\| \hbar \nabla_x   a_{\hbar,P}^\pm \|_{L^2}) + \| \nabla_x v_\pm (P,\cdot) \|_{L^\infty}   \Big)^2
\end{equation}
it follows (when $R > |P|$)
\begin{eqnarray}
\label{s-UR2}
|\mathbb{P}^\pm_\hbar  (\Bbb T^n \times U_R)|  \le  \sum_{\alpha \in \mathbb{Z}^n,  |\hbar \alpha| >R}   \frac{ C_{n,P} }{|- \hbar \alpha + P|^2} \le \int_{\Bbb R^n / B_R (0)}   \frac{ C_{n,P} }{|- y + P|^2} dy
\end{eqnarray}
The last ($\hbar$-independent) upper bound implies that 
\begin{equation}
\lim_{R \rightarrow + \infty} \,  \sup_{0 < \hbar \le 1} \, |\mathbb{P}^\pm_\hbar  (\Bbb T^n \times U_R)|   = 0.
\end{equation}
\end{proof}

We next exhibit  a property of the involved monokinetic measures.
\begin{proposition}
\label{MM}
Let  $dm_P^\pm$ as in  (\ref{def-dm}) and $v_- (P,\cdot) \in C^{0,1} (\Bbb T^n;\Bbb R)$ be a weak KAM solution of negative type for the H-J equation (\ref{def-eff0}).   Define the lifted Borel measure on $\Bbb T^n \times \Bbb R^n$ by 
\begin{equation}
\label{mono}
\int_{\Bbb T^n \times \Bbb R^n}  \phi (x,\eta) d\widetilde{m}_P^\pm (x,\eta)  :=   \int_{\Bbb T^n \times \Bbb R^n}  \phi (x,P+ \nabla_x v_- (P,x)) dm_P^\pm (x), \quad  \forall \ \phi \in A.
\end{equation}   
Then,   $d\widetilde{m}_P^\pm$ does not depend on the choice of $v_- (P,\cdot)$, namely 
\begin{equation}
\int_{\Bbb T^n \times \Bbb R^n}  \phi (x,\eta) d\widetilde{m}_P^\pm (x,\eta)  =   \int_{\Bbb T^n}  \phi (x,P+ \nabla_x v_-^{\prime} (P,x)) dm_P^\pm (x)
\end{equation} 
for any other weak KAM of negative type $v_-^{\prime} (P,x)$. Moreover, for any weak KAM  of positve type $v_{+} (P,x)$ it holds
\begin{equation}
\int_{\Bbb T^n \times \Bbb R^n}  \phi (x,\eta) d\widetilde{m}_P^\pm (x,\eta)  =   \int_{\Bbb T^n}  \phi (x,P+ \nabla_x v_+ (P,x)) dm_P^\pm (x)
\end{equation} 
Finally, there exists a Borel measurable function $g^\pm (P,\cdot)    : \Bbb T^n \rightarrow \Bbb R^+$ such that
\begin{equation}
\label{abs}
\int_{\Bbb T^n \times \Bbb R^n}  \phi (x) dm_P^\pm (x)  =   \int_{\Bbb T^n}  \phi (x)  \  g^\pm  (P,x)    d\sigma_P^\pm (x).
\end{equation} 
\end{proposition}
\begin{proof}
For any $v_\pm (P,\cdot) \in C^{0,1} (\Bbb T^n;\Bbb R)$ which is a weak KAM solution of Hamilton-Jacobi equation (\ref{def-eff0}), the map $x \mapsto \nabla_x v_\pm (P,x)$ is continuous  and uniformely bounded on its domain of definition ${\rm dom}(\nabla_x v_\pm (P,\cdot)) \subseteq \Bbb T^n$. Moreover, since we assumed $dm_P^\pm  \ll  d\sigma_P$ then ${\rm supp }(dm_P^\pm ) \subseteq {\rm supp }( d\sigma_P)$. 
By recalling that  ${\rm supp }( d\sigma_P)  \subseteq    \pi(\mathcal{M}_P^\star)  \subseteq    \pi(\mathcal{A}_P^\star)$ and thanks to the localization the Aubry set $\mathcal{A}_P^\star$ shown in Section \ref{sec-Au}, it follows
 \begin{equation}
 \int_{\Bbb T^n \times \Bbb R^n}  \phi (x,\eta) d\widetilde{m}_P^\pm (x,\eta) =   \int_{\Bbb T^n}  \phi (x,P+ \nabla_x v_{\pm} (P,x)) dm_P^\pm (x)
 \end{equation}
for  any $v_{\pm} (P,\cdot) \in C^{0,1} (\Bbb T^n;\Bbb R)$   weak KAM solutions of Hamilton-Jacobi equation. Finally, the assumption on the absolute continuity of $dm_P^\pm$ with respect to $ d\sigma_P$ together with the well known Radon-Nikodym derivative  provides the existence of  $g^\pm  (P,x)$ satisfying  (\ref{abs}).
\end{proof}

\begin{lemma}
\label{Lem-ac}
Let   
\begin{equation}
\label{mP}
d\widetilde{m}_P^\pm (x,\eta) := \delta (\eta-P-\nabla_x v_\pm (P,x)) dm_P^\pm (x) 
\end{equation} 
be as in Proposition \ref{MM}. Then, $d\widetilde{m}_P^\pm$ is absolutely continuous to $dw_P$ the Legendre transform of a Mather P-minimal measure. In particular,  there exists a Borel measurable function $g^\pm (P,\cdot)    : \Bbb T^n \rightarrow \Bbb R^+$ such that 
\begin{equation}
\label{mP2}
d\widetilde{m}_P^\pm (x,\eta) = g^\pm (P,x) dw_P (x,\eta)
\end{equation} 
where $dw_P (x,\eta) =  \delta (\eta-P-\nabla_x v_\pm (P,x)) d\sigma_P (x) $.
\end{lemma}
\begin{proof}
By  the assumption within Definition \ref{wave-d}, it holds  $dm_P \ll \pi_\star(dw_P) =: d\sigma_P$ where $dw_P$ is Legendre transform of a Mather $P$-minimal measure  $d\mu_P$ as in (\ref{def-dwP}). Equivalently, we can also take $\pi_\star (d\mu_P) =: d\sigma_P$  since the push forward by the canonical projection is the same.
Thus, there exists a Borel measurable function $g^\pm (P,\cdot)    : \Bbb T^n \rightarrow \Bbb R^+$ such that
\begin{equation}
d\widetilde{m}_P^\pm (x,\eta) = g^\pm (P,x)  \delta (\eta-P-\nabla_x v_\pm (P,x)) d\sigma_P (x). 
\end{equation} 
In fact,  it holds the equality $\delta (\eta-P-\nabla_x v_\pm (P,x)) d\sigma_P (x) = dw_P (x,\eta)$ thanks to the inclusion 
$$
{\rm supp}(dw_P) \subseteq  \mathcal{A}_P^\star \subseteq {\rm Graph} (P+\nabla_x v_\pm (P,\cdot \,)),
$$
see Lemma 3.1 shown in \cite{F-G-S}. The (\ref{mP2}) follows directly.
\end{proof}

We are now ready to provide the result involving the semiclassical limits of the Wigner transform for the above class of WKB-type wave functions.
\begin{theorem}
\label{TH4}
Let $P \in \ell \, \Bbb Z^n$ for some $\ell >0$, $\hbar^{-1} \in \ell^{-1} \Bbb N$, $v_\pm$ be weak KAM solutions of H-J equation (\ref{def-eff0}) and $\varphi_{\hbar}^{\pm}$ be the associated WKB wave functions as in Def. \ref{wave-d}, $dm_P^\pm$ as in  Def. \ref{wave-d}.  Then,   
\begin{equation}
\label{WKB-c}
\lim_{\hbar \rightarrow 0^+} W_\hbar \varphi_{\hbar}^{\pm} (x,\eta) =  \delta (\eta-P-\nabla_x v_\pm (P,x)) dm_P^\pm (x) =:   d\widetilde{m}_P^\pm (x,\eta)
\end{equation} 
in $A^\prime$ and passing through a subsequence.
\end{theorem}
\begin{proof}
The Wigner transform in the variables $(q,p) \in \Bbb Z^n \times \Bbb R^n$:
\begin{eqnarray}
\widehat{W}_\hbar  \varphi_{\hbar}^\pm (q,p)  &:=& 
\int_{\mathbb{T}^n} \varphi_{\hbar}^\pm (y)^\star   e^{i (q \cdot y + \hbar p \cdot q / 2) }  \varphi_{\hbar}^\pm  (y+\hbar p)   dy
\nonumber\\
&=&    \int_{\mathbb{T}^n}  e^{i [ \hbar p \cdot q / 2 + P \cdot  p] } \  e^{i q \cdot y}   e^{ \frac{i}{\hbar} [v_\pm (P ,y+ \hbar p) - v_\pm (P ,y)   ]   }  a_{\hbar,P}^\pm (y)  a_{\hbar,P}^\pm (y+\hbar p)   dy.
\label{W-qp}
\end{eqnarray}
Since  $a_{\hbar,P}^\pm \in L^{2}(\Bbb T^n; \Bbb R^+)$,  the integral in (\ref{W-qp}) is absolutely convergent and the function $\widehat{W}_\hbar \varphi_\hbar^\pm (\cdot)$ is Lebesgue measurable and uniformely bounded in both variables.\\ 
By the $H^1$-regularity we can write 
$
a_{\hbar,P}^\pm (y+\hbar p)  =   a_{\hbar,P}^\pm (y) + \hbar \int_{0}^1 p \cdot  \nabla_x a_{\hbar,P}^\pm (y + \lambda \hbar p)  d\lambda  
$
and hence 
\begin{eqnarray}
\label{bound-H1}
\|  a_{\hbar,P}^\pm  (\diamond +\hbar p) -   a_{\hbar,P}^\pm (\diamond) \|_{L^2}  \le   |p| \ \hbar \ \|  a_{\hbar,P}^\pm \|_{H^1}
\end{eqnarray}
Thus,
\begin{eqnarray}
\widehat{W}_\hbar  \varphi_{\hbar}^\pm (q,p)  =    \int_{\mathbb{T}^n}  e^{i [ \hbar p \cdot q / 2 + P \cdot  p] } \  e^{i q \cdot y}   e^{ \frac{i}{\hbar} [v_\pm (P ,y+ \hbar p) - v_\pm (P ,y)   ]   }  a_{\hbar,P}^\pm (y)^2    dy + R_\hbar (q,p)
\end{eqnarray}
where
$$
R_\hbar (q,p) :=  \int_{\mathbb{T}^n}  e^{i [ \hbar p \cdot q / 2 + P \cdot  p] } \  e^{i q \cdot y}   e^{ \frac{i}{\hbar} [v_\pm (P ,y+ \hbar p) - v_\pm (P ,y)   ]   }  a_{\hbar,P}^\pm (y)  [a_{\hbar,P}^\pm (y+\hbar p)  -  a_{\hbar,P}^\pm (y)  ]  dy
$$
and thus $\forall (q,p) \in \Bbb Z^n \times \Bbb R^n$
\begin{eqnarray}
\label{est-R}
|R_\hbar (q,p)| \le   {\rm vol}(\Bbb T^n)  \|  a_{\hbar,P} \|_{L^2} \|  a_{\hbar,P} (\diamond +\hbar p) -   a_{\hbar,P} (\diamond) \|_{L^2}     \le  (2\pi)^n \  |p| \ \hbar \ \|  a_{\hbar,P} \|_{H^1}.
\end{eqnarray}
%\begin{eqnarray}
%v(P_\hbar ,y+ \hbar p) - v(P_\hbar ,y)   = \hbar \int_0^1  \langle  \nabla_x v (P_\hbar ,y + \lambda \hbar p) , p \rangle \, d\lambda, 
%\end{eqnarray}
%where $\nabla_x v (P_\hbar ,y)$ is defined $dy$-a.e. on $\Bbb T^n$, and thus  the set of $p \in \Bbb R^n$ such that $ \nabla_x v (P_\hbar ,y + \lambda \hbar p)$  is not defined  $d\lambda$-a.e. has zero Lebesgue measure.  
For any $\phi \in A$ and ${\rm supp}(\phi)$ is compact, 
\begin{eqnarray}
\label{mean-Op}
&& \sum_{q \in \mathbb{Z}^n}  \int_{\mathbb{R}^n}  \widehat{\phi} (q,p)  \widehat{W}_\hbar  \varphi_{\hbar}^\pm (q,p)  (q,p) dp 
\\
&=& 
\sum_{q \in \mathbb{Z}^n}  \int_{\mathbb{R}^n}  \widehat{\phi} (q,p)   \int_{\mathbb{T}^n}  e^{i [ \hbar p \cdot q / 2 + P \cdot  p] } \  e^{i q \cdot y}   e^{ \frac{i}{\hbar}  [v_\pm (P ,y+ \hbar p) - v_\pm (P ,y)   ]     }  a_{\hbar,P} (y)^2    dy dp  
\nonumber
\\
&+& \sum_{q \in \mathbb{Z}^n}  \int_{\mathbb{R}^n}  \widehat{\phi} (q,p)  R_\hbar (q,p) dp.
\end{eqnarray}
An easy computation shows that
$$
\Big| \sum_{q \in \mathbb{Z}^n}  \int_{\mathbb{R}^n}  \widehat{\phi} (q,p)  R_\hbar (q,p) dp \Big| \le  \sum_{q \in \mathbb{Z}^n}  \int_{\mathbb{R}^n}  |\widehat{\phi} (q,p)| (2\pi)^n \  |p| \ \hbar \ \|  a_{\hbar,P} \|_{H^1} dp
$$
and hence, since ${\rm supp} ( \widehat{\phi})$ is compact and $\hbar \|  a_{\hbar,P}^\pm   \|_{H^{1}} \longrightarrow 0$ as $\hbar \longrightarrow 0^+$ (see Remark \ref{EX-a}) it follows
\begin{equation}
\label{est-R2}
(2\pi)^n \sum_{q \in \mathbb{Z}^n}  \int_{\mathbb{R}^n}  |\widehat{\phi} (q,p)|  \  |p|  dp  \ \hbar \ \|  a_{\hbar,P}^\pm \|_{H^1} \longrightarrow 0^+ \quad {\rm as } \quad \hbar \longrightarrow 0^+.
\end{equation}
In view of (\ref{est-R2}) and the compactness of ${\rm supp} ( \widehat{\phi})$, the (\ref{mean-Op}) reads
\begin{eqnarray} 
\label{4124}  
\sum_{q \in \mathbb{Z}^n}  \int_{\mathbb{R}^n}  \widehat{\phi} (q,p)   \lim_{\hbar \rightarrow 0^+}  \int_{\mathbb{T}^n}  e^{i [ \hbar p \cdot q / 2 + P \cdot  p] } \  e^{i q \cdot y}   e^{ \frac{i}{\hbar}  [v_\pm (P ,y+ \hbar p) - v_\pm (P ,y)   ]     }  |a_{\hbar,P}^\pm (y)|^2    dy dp  .
\nonumber
\label{cat5}
\end{eqnarray}
By looking at the integral
\begin{equation}
\label{410}
\int_{\mathbb{T}^n}  e^{i [ \hbar p \cdot q / 2 + P \cdot  p] } \  e^{i q \cdot y}   e^{ \frac{i}{\hbar}  [v_\pm (P ,y+ \hbar p) - v_\pm (P ,y)   ]     }  
|a_{\hbar,P}^\pm (y)|^2    dy
\end{equation}
we observe that $e^{i ( \hbar p \cdot q / 2) }  e^{ \frac{i}{\hbar}  [v_\pm (P ,y+ \hbar p) - v_\pm (P ,y)   ]     } $ is a family of uniformely bounded continuous functions on $\Bbb T^n$ such that 
\begin{equation}
\label{412}
\lim_{\hbar \rightarrow 0^+}e^{i ( \hbar p \cdot q / 2) }   e^{ \frac{i}{\hbar}  [v_\pm (P ,y+ \hbar p) - v_\pm (P ,y)   ]     }  = e^{ i p \cdot \nabla_x v_\pm (P,y)}
\end{equation}
$\forall (q,p) \in {\rm supp} ( \widehat{\phi})$ and $\forall$ $y \in {\rm dom}(\nabla_x v_\pm (P,\cdot))$, since any map $x \longmapsto \nabla_x v_{\pm}(P,x)$  is continuous on ${\rm dom}(\nabla_x v_{\pm} (P,\cdot))$ (as we recall in Section  \ref{sub-wHJ}). By the inclusions 
\begin{equation}
{\rm supp}(dm_{P}^\pm) \subseteq {\rm supp}(d\sigma_{P}) \subseteq {\rm dom}(\nabla_x v_\pm (P,\cdot))
\end{equation}
we deduce that (\ref{412}) is not fulfilled only for a set of zero $dm_{P}^\pm$ measure.\\		
Hence, we can apply   Lemma \ref{LemmaYang} for the semiclassical limits of the  integral (\ref{410}) to obtain
\begin{equation}
 \int_{\mathbb{T}^n}  e^{i  P \cdot  p } \  e^{i q \cdot y}   e^{ i p \cdot \nabla_x v_\pm (P,y)     }  dm_P^\pm (y) dp.
\end{equation}
We deduce that the semiclassical limits of the mean value  (\ref{4124}) read
\begin{eqnarray}
&& \sum_{q \in \mathbb{Z}^n}  \int_{\mathbb{R}^n}  \widehat{\phi} (q,p) \Big(  \int_{\mathbb{T}^n}  e^{i  P \cdot  p } \  e^{i q \cdot y}   e^{ i p \cdot \nabla_x v_\pm (P,y)     }  dm_P^\pm (y) \Big) dp.
\\
&& \ \ =  \int_{\mathbb{T}^n} \sum_{q \in \mathbb{Z}^n}  \int_{\mathbb{R}^n}  \widehat{\phi} (q,p)    e^{i  P \cdot  p } \  e^{i q \cdot y}   e^{ i p \cdot \nabla_x v_\pm (P,y)     }  dp \ dm_P^\pm (y) 
\end{eqnarray}
where we used again the compacteness of ${\rm supp}( \widehat{\phi})$.
Through the inverse phase-space Fourier transform  the above expression becomes
\begin{eqnarray}
\int_{\mathbb{T}^n}  \phi (y, P +  \nabla_x v_\pm (P,y) ) \,  dm_P^\pm (y).
\end{eqnarray}
\end{proof}

\begin{remark}
Let $P \in \ell \, \Bbb Z^n$ for some $\ell >0$ and $\varphi_\hbar^\pm$ as in Definition \ref{wave-d}.   Define the current 
\begin{equation}
\label{def-J}
J_\hbar^\pm (x) :=  \hbar \, {\rm Im} (  (\varphi_\hbar^\pm)^\star  \nabla_x \varphi_\hbar^\pm (x) ) = (P + \nabla_x v_\pm (P,x))  |a_{\hbar,P}^\pm (x)|^2 
\end{equation}
The (formal) free current equation ${\rm div}_x J_\hbar^\pm (x) = 0$ becomes well-posed in the weak sense: 
\begin{eqnarray}
\int_{\Bbb T^n} \nabla_x f (x) \cdot  J_\hbar^\pm (x) \ dx =  0 \quad \quad \forall f \in C^\infty (\Bbb T^n;\Bbb R).
\end{eqnarray}
In particular, we recall the inclusion  (\ref{inc-G}) which implies, together with the assumptions on $a_{\hbar,P}^\pm$,  the estimate $\sup_{0 < \hbar \le1} \| J_\hbar^\pm \|_{L^1} \le   \| P + \nabla_x v_\pm (P,\cdot)   \|_{L^\infty}    < + \infty$.  However, the low regularity $v_\pm (P,\cdot) \in C^{0,1} (\Bbb T^n;\Bbb R^n)$ do not guarantees the existence of some amplitude function satifying this equation, hence we  write the asymptotic condition 
\begin{eqnarray}
\Big|  \int_{\Bbb T^n} \nabla_x f (x) \cdot  J_{\hbar_j}^\pm (x) \ dx \Big|   \longrightarrow  0, \quad \forall f \in C^\infty (\Bbb T^n;\Bbb R)  
\end{eqnarray}
for a sequence $\{ \hbar_j^{-1} \}_{j \in \Bbb N} \in \ell^{-1} \Bbb N$ with $\hbar_j \longrightarrow 0^+ $ as $j \longrightarrow + \infty$. 
\end{remark}

The above observations become meaningful in view of the following result.
\begin{proposition}
\label{prop46}
Let $P \in \ell \, \Bbb Z^n$ for some $\ell >0$, $v_\pm (P,\cdot) \in C^{0,1} (\Bbb T^n;\Bbb R)$  be a weak KAM solution  for (\ref{def-eff0}). Then, there exist 
 $a_{\hbar,P}^\pm $ as  in Remark \ref{EX-a} such that 
 the (unique) weak-$\star$ limit $dm_P (x) := \lim_{j \rightarrow + \infty} |a_{\hbar_j ,P}^\pm (x)|^2dx$ equal $d\sigma_P := \pi_\star (dw_P) $ where $dw_P$ is the Legendre transform of a Mather $P$-minimal measure and 
\begin{equation}
\label{divJ1}
\Big| \int_{\Bbb T^n} \nabla_x f (x) \cdot  J_{\hbar_j}^\pm (x)   dx \Big| \longrightarrow  0 \quad {\rm as} \ j \longrightarrow + \infty \quad  \forall f \in C^\infty (\Bbb T^n;\Bbb R).
\end{equation}
\end{proposition}
\begin{proof}
Let  $d\sigma_P := \pi_\star (dw_P) = d\mu_P$ with $dw_P$ as in (\ref{def-dwP}). Then, 
$d\sigma_P$ is a Borel probability measure $\Bbb T^n$ with 
\begin{equation}
\label{incl-s}
{\rm supp}(d\sigma_P) \subseteq \pi_\star (\mathcal{M}_P^\star)  \subseteq \pi_\star (\mathcal{A}_P^\star) \subseteq {\rm dom}(\nabla_x v_\pm (P,\cdot \,)) . 
\end{equation} 
Moreover,  it holds
\begin{equation}
\label{f-div}
\int_{\Bbb T^n} \nabla_x f (x) \cdot (P+ \nabla_x v_\pm (P,x)) \  d\sigma_P (x) = 0 \quad \forall f \in C^\infty (\Bbb T^n;\Bbb R). 
\end{equation}
Indeed, $dw_P := \mathcal{L}_\star (d\mu_P)$ and $d\mu_P$ is invariant under Lagrangian flow, hence closed, which means that
$$
\int_{\Bbb T^n \times \Bbb R^n} \nabla_x f (x) \cdot \xi  \ d\mu_P (x,\xi) = 0 \quad \forall f \in C^\infty (\Bbb T^n;\Bbb R). 
$$  
Here the Lagrangian reads $L(x,\xi) = \frac{1}{2} |\xi|^2 + V(x)$ and thus the Legendre transform  $\mathcal{L}(x,\xi) = (x,\xi)$, which gives 
$$
\int_{\Bbb T^n \times \Bbb R^n} \nabla_x f (x) \cdot \eta  \ dw_P (x,\eta) = 0 \quad \forall f \in C^\infty (\Bbb T^n;\Bbb R). 
$$ 
By Lemma 3.1 in \cite{F-G-S}, we have necessary ${\rm supp}(dw_P) \subseteq  \mathcal{A}_P^\star \subseteq {\rm Graph}(P+\nabla_x v_\pm (P,\cdot))$. Thus, we can restrict 
$dw_P|_{{\rm Graph}(P+\nabla_x v_\pm (P,\cdot))}$ since ${\rm Graph}(P+\nabla_x v_\pm (P,\cdot))$ are  Borel measurable subsets of $\Bbb T^n \times \Bbb R^n$ containing the support of this measure. Hence
$$
\int_{{\rm Graph}(P+\nabla_x v_\pm (P,\cdot))} \nabla_x f (x) \cdot \eta  \ dw_P (x,\eta) = 0 \quad \forall f \in C^\infty (\Bbb T^n;\Bbb R). 
$$ 
The canonical projection $\pi : {\rm Graph}(P+\nabla_x v_\pm (P,\cdot)) \rightarrow \Bbb T^n$ is a Borel measurable map, because of $\overline{ {\rm Graph}(P+\nabla_x v_\pm (P,\cdot))} = \Bbb T^n$. We can apply the change of variables and get (\ref{f-div}).\\
Now,  define the Borel probability measure $dm_P (x) := d\sigma_P (x)$ on $\Bbb T^n$. Recalling Remark \ref{EX-a}, there exists $a_{\hbar,P}^\pm   \in C^{k}(\Bbb T^n; \Bbb R^+)$ such that   $ \lim_{\hbar_j \rightarrow 0^+} |a_{\hbar_j ,P}^\pm (x)|^2 =  dm_P (x)$ in the weak-$\star$ convergence of Borel measures on $\Bbb T^n$. Notice that now we do not write $dm_P$  as $dm_P^\pm$ since in fact holds the inclusion (\ref{incl-s}).\\ 
Thus, we look at 
\begin{eqnarray}
\int_{\Bbb T^n} \nabla_x f (x) \cdot  J_\hbar^\pm (x)  dx  =  \int_{\Bbb T^n} \nabla_x f (x) \cdot (P + \nabla_x v_\pm (P ,x)) \  |a_{\hbar,P}^\pm (x)|^2 dx.
%\nonumber\\
%&+& \int_{\Bbb T^n} \nabla_x \phi (x) \cdot \hbar \nabla_x a_{\hbar,P}^\pm(x)  \  a_{\hbar,P}^\pm(x) g^{-1}(P,x) dx.
\label{J-dec}
\end{eqnarray}
and  observe that the function  
$$
x  \longmapsto \nabla_x f(x) \cdot (P + \nabla_x v_\pm (P,x)) 
$$
is a bounded Borel measurable function, and $x \mapsto \nabla_x v_\pm (P,x)$ is continuous on its domain of definition. Hence, the set of $x \in \Bbb T^n$ such that $\exists$ $\{x_k \}_{k \in \mathbb{N}} \subset \Bbb T^n,\,\,\lim_{k \to+\infty} x_{k} = x$ and 
$$
\lim_{k \rightarrow + \infty}   \nabla_x f (x_k) \cdot (P + \nabla_x v_\pm (P, x_k))\   \neq  \nabla_x f (x) \cdot (P + \nabla_x v_\pm (P,x))\  
$$
is a set of zero $dm_P$-measure.  We now apply Lemma \ref{LemmaYang} to get
\begin{eqnarray}
&&\lim_{j \rightarrow + \infty} \int_{\Bbb T^n} \nabla_x f (x) \cdot  (P + \nabla_x v_\pm (P, x))\   |a_{\hbar_j,P}^\pm (x)|^2  dx 
\\
&=&    \int_{\Bbb T^n} \nabla_x f (x) \cdot  (P + \nabla_x v_\pm (P, x))\   dm_P (x) = 0
\end{eqnarray}
where the last equality is given by the above setting of $dm_P (x) := d\sigma_P (x)$ and (\ref{f-div}).
\end{proof}

\section{Propagation of Wigner measures  on weak KAM tori}\label{pwt}   \markboth{Propagation of  Wigner measures  on weak KAM tori}{}
\subsection{The forward and backward propagation}
The main result of the section reads as 

\begin{theorem}
\label{th51}
Let $\varphi_{\hbar}^{\pm}$ be as in Def. \ref{wave-d} and $\psi_\hbar (t) := e^{-\frac{i}{\hbar}  Op_\hbar^w (H) t} \varphi_\hbar$. Let $\widetilde{m}_P^\pm (t)$
 be a limit of  $W_\hbar \psi_\hbar (t)$ in $L^\infty ([-T,+T]; A^\prime)$,  and $\widetilde{m}_P^\pm $, $g_\pm (P,x)$ be as in Proposition \ref{MM}. Then,  $\widetilde{m}_P^\pm (t) = (\varphi_H^t)_\star (\widetilde{m}_P^\pm) \in \mathcal{M}^{1+} (\Bbb T^n \times \Bbb R^n)$. Moreover, $\forall \phi \in A$ and  $\forall$ $t \ge 0$
\begin{eqnarray}
\int_{\Bbb T^n \times \Bbb R^n}  \phi (x,\eta) \  d\widetilde{m}_P^+ (t,x,\eta)  &=& \int_{\Bbb T^n}  \phi (x,P+ \nabla_x v_{+} (P,x))  \  {\bf g}_+ (t,P, x)    d\sigma_P(x)
\\
{\bf g}_+ (t,P, x) &:=&  g_+ (P, \pi \circ \varphi_H^{-t} (x,P+ \nabla_x v_{-} (P,x)))
\label{defg+}
\end{eqnarray}
Whereas $\forall$ $ t\le 0$
\begin{eqnarray}
\int_{\Bbb T^n \times \Bbb R^n}  \phi (x,\eta)  \  d\widetilde{m}_P^- (t,x,\eta)  &=& \int_{\Bbb T^n}  \phi (x,P+ \nabla_x v_{-} (P,x))  \  {\bf g}_- (t,P, x)    d\sigma_P(x)
\\
{\bf g}_- (t,P, x) &:=&  g_- (P, \pi \circ \varphi_H^{-t} (x,P+ \nabla_x v_{+} (P,x)))
\label{defg-}
\end{eqnarray}
\end{theorem}
\begin{proof}
%By Theorem \ref{TH41} the time dependent $d\widetilde{m}_P^\pm (t)$, by assumption the limit of the Wigner transform $W_\hbar \psi_\hbar (t)$ in $L^\infty ([-T,+T]; A^\prime)$,  solves the Liouville equation in the distributional sense $L^\infty ([-T,+T]; A^\prime)$ and hence, thanks to the uniqueness for the solutions of this continuity equation with smooth coefficients, it holds $d\widetilde{m}_P^\pm (t) = (\varphi_H^t)_\star (d\widetilde{m}_P^\pm (0))$. This also implies that $d\widetilde{m}_P^\pm (\cdot) \in C ([-T,+T]; A^\prime)$. 
%On the other hand, for our initial data $\varphi_{\hbar}^{\pm}$ we proved, within Theorem \ref{TH4}, that the  Wigner transform $W_\hbar \varphi_\hbar^\pm$  is weak converging (for test functions in A) to the monokinetic the probability measure $d\widetilde{m}_P^\pm \in \mathcal{M}^{1+} (\Bbb T^n \times \Bbb R^n)$. Moreover, recalling Lemma \ref{L-T1}, the complex measure $\Bbb P_\hbar^\pm$ is tight  and hence the time evolution $\Bbb P_\hbar^\pm (t)$ is tight as well (see Proposition \ref{Prop-T}).  This implies that there exists limits of  $\Bbb P_\hbar^\pm(t)$ in the sense of (\ref{NC}), namely there exist weak limits of $W_\hbar \psi_\hbar (t)$ with respect to test functions in $C_b (\Bbb T^n \times \Bbb R^n) \supset A$ to some (a priori complex) Borel probability measures for any fixed $t$. In fact, this means that it must be  $d\widetilde{m}_P^\pm (t) = (\varphi_H^t)_\star (d\widetilde{m}_P^\pm = d\widetilde{m}_P^\pm (0)) \in \mathcal{M}^{1+} (\Bbb T^n \times \Bbb R^n)$.\\   
%Next, we underline that $\forall \phi,\psi \in A$

By Theorem \ref{TH41},  any  limit $d{\rm w} (t)$ of the Wigner transform $W_\hbar
\psi_\hbar (t)$ in $L^\infty ([-T,+T]; A^\prime)$  solves the Liouville equation in
the distributional sense $L^\infty ([-T,+T]; A^\prime)$ and hence, thanks to the
uniqueness for the solutions of this continuity equation, it holds $d{\rm w}  (t) =
(\varphi_H^t)_\star (d {\rm w}  (0))$. This also implies that $d{\rm w}  (\cdot) \in
C ([-T,+T]; A^\prime)$. 
On the other hand, for our initial data $\varphi_{\hbar}^{\pm}$ we proved, within
Theorem \ref{TH4}, that the  Wigner transform $W_\hbar \varphi_\hbar^\pm$  is weak
converging (for test functions in A) to the monokinetic probability measures
$d\widetilde{m}_P^\pm \in \mathcal{M}^{1+} (\Bbb T^n \times \Bbb R^n)$. Moreover,
recalling Lemma \ref{L-T1}, the complex measures $\Bbb P_\hbar^\pm$ are tight  and
hence their time evolution $\Bbb P_\hbar^\pm (t)$ is tight as well (see Proposition
\ref{Prop-T}).  This implies that there exist semiclassical limits of  $\Bbb
P_\hbar^\pm(t)$ in the sense of (\ref{NC}), namely there exist weak limits of
$W_\hbar \psi_\hbar (t)$ with respect to test functions in $C_b (\Bbb T^n \times
\Bbb R^n) \supset A$ to some (a priori complex) Borel probability measures for any
fixed $t$. In fact, this means that it must be  $d{\rm w}  (t) = (\varphi_H^t)_\star
( d {\rm w} (0) = d\widetilde{m}_P^\pm ) \in \mathcal{M}^{1+} (\Bbb T^n \times \Bbb
R^n)$.
From now on, we write $d\widetilde{m}_P^\pm (t) := (\varphi_H^t)_\star
(d\widetilde{m}_P^\pm)$.\\ 
Next, we underline that $\forall \phi,\psi \in A$
\begin{eqnarray}
\int_{\Bbb T^n \times \Bbb R^n}  \phi (x,\eta) \  d\widetilde{m}_P^\pm (t,x,\eta)  &=&  \int_{\Bbb T^n \times \Bbb R^n}  \phi \circ \varphi_H^{t} (x,\eta) \  d\widetilde{m}_P^\pm (x,\eta)
\\
\int_{\Bbb T^n \times \Bbb R^n}  \psi (x,\eta) \  d\widetilde{m}_P^\pm  (x,\eta)  &=&   \int_{\Bbb T^n}  \psi (x,P+ \nabla_x v_{\pm} (P,x))  \  g_\pm (P,x)  d\sigma_P(x).
\end{eqnarray}
Hence
\begin{equation}
\label{108}
\int_{\Bbb T^n \times \Bbb R^n}  \phi (x,\eta) \  d\widetilde{m}_P^\pm  (t,x,\eta) = \int_{\Bbb T^n}  \phi \circ \varphi_H^t (x,P+ \nabla_x v_{\pm} (P,x)) \  g_\pm (P,x)  d\sigma_P(x).
\end{equation}
%We now recall that $d\sigma_P := \pi_\star (dw_P)$ where the Mather P-minimal measure $dw_P$ takes the monokinetic form
We now recall that $d\sigma_P := \pi_\star (dw_P)$  where $dw_P$ is the Legendre
transform of a Mather P-minimal measure, which takes the monokinetic form
\begin{equation}
\label{form-dw}
\int_{\Bbb T^n \times \Bbb R^n}  \phi (x,\eta) \  dw_P(x,\eta) = \int_{\Bbb T^n}  \phi (x,P+ \nabla_x v_{\pm} (P,x)) \   d\sigma_P(x)
\end{equation}
and $dw_P$ is invariant under the Hamiltonian flow.  This is a consequence of the Lemma 3.1 in \cite{F-G-S}, which gives ${\rm supp}(dw_P) \subseteq  \mathcal{A}_P^\star$ and thanks to the inclusion  $\mathcal{A}_P^\star \subseteq {\rm Graph}(P+\nabla_x v_\pm (P,\cdot))$.\\ 
Hence, we can rewrite
\begin{equation}
\label{1fg}
\int_{\Bbb T^n \times \Bbb R^n}  \phi (x,\eta) \  d\widetilde{m}_P^\pm  (t,x,\eta) = \int_{\Bbb T^n \times \Bbb R^n}  \phi \circ \varphi_H^t (x,\eta) \  g_\pm (P,\pi(x,\eta))  dw_P(x,\eta).
\end{equation}
%Since ${\rm supp}( dw_P) \subseteq G_{P}^+ := {\rm Graph}(P+\nabla_x v_+ (P,\cdot \, ))$ and  $\pi \ \overline{G_{P}^+} = \Bbb T^n$, it follows that....\\ 
%$\varphi_H^t {\rm Graph}(P+\nabla_x v_+ (P,\cdot \, )) \subseteq {\rm Graph}(P+\nabla_x v_+ (P,\cdot \, ))$ $\forall t\ge0$, the one parameter semigroup of Borel measurable maps on $\Bbb T^n$:
By the generalized change of variables,
\begin{equation}
\label{13fg}
\int_{\Bbb T^n \times \Bbb R^n}  \phi (x,\eta) \  d\widetilde{m}_P^\pm  (t,x,\eta) = \int_{\Bbb T^n \times \Bbb R^n}  \phi (x,\eta) \  g_\pm (P,\pi \circ \varphi_H^{-t} (x,\eta))  (\varphi_H^{-t})_\star dw_P(x,\eta)
\end{equation}
and thanks to the invariance of $dw_P$,
\begin{equation}
\int_{\Bbb T^n \times \Bbb R^n}  \phi (x,\eta) \  d\widetilde{m}_P^\pm (t,x,\eta) = \int_{\Bbb T^n \times \Bbb R^n}  \phi (x,\eta) \  g_\pm (P,\pi \circ \varphi_H^{-t} (x,\eta))   dw_P(x,\eta).
\end{equation}
 By (\ref{form-dw})
 \begin{equation}
\label{rep-dwt} 
\int_{\Bbb T^n \times \Bbb R^n}  \phi (x,\eta) \  d\widetilde{m}_P^\pm (t,x,\eta) = \int_{\Bbb T^n}  \phi (x,P+ \nabla_x v_{\pm} (P,x)) \  g(P,\pi \circ \varphi_H^{-t} (x,P+ \nabla_x v_{\pm} (P,x)))  \ d\sigma_P(x). 
\end{equation}
Thus, we can define 
 \begin{equation}
 \label{g+1}
{\bf g}_+ (t,P, x) :=  g_+ (P, \pi \circ \varphi_H^{-t} (x,P+ \nabla_x v_{-} (P,x))) \quad  {\rm for} \ t\ge 0
 \end{equation}
and 
 \begin{equation}
{\bf g}_{-} (t,P, x) :=  g_- (P, \pi \circ \varphi_H^{-t} (x,P+ \nabla_x v_{+} (P,x))) \quad  {\rm for} \ t\le 0.
 \label{g-1}
 \end{equation}
\end{proof}

\begin{remark}
We notice that the supports of the measures $d\widetilde{m}_P^\pm (t)$ are contained, for any $t \in \Bbb R$, in the Mather set $\mathcal{M}_P^\star \subseteq \mathcal{A}_P^\star$ in the phase space  which is invariant under the Hamiltonian flow as well as $\mathcal{A}_P^\star$. Hence, these are also contained in any set ${\rm Graph}(P+\nabla_x v_\pm (P,\cdot))$ and this means that we could write   several possible equilvalent Borel measurable density functions ${\bf g}_\pm (t,P, x)$. However, within the next result we underline that the functions ${\bf g}_+$ solve a forward continuity equation with respect to the vector field  $P+\nabla_x v_+ (P,\cdot)$  and  ${\bf g}_-$ solve a backward equation with respect to $P+\nabla_x v_- (P,\cdot)$. 
\end{remark}

\begin{proposition}
 \label{prop53}
Let ${\bf g}_\pm$ and $d \sigma_P$ as in Theorem \ref{th51}. Then, $\forall f \in C^\infty ([0,t] \times \Bbb T^n;\Bbb R)$
\begin{equation}
\label{redc+}
\int_0^t \int_{\Bbb T^n} [\partial_s f(s,x)  +   \nabla_x f(s,x) \cdot (P+ \nabla_x v_{+} (P,x))] \ {\bf g}_+ (s,P,x) d \sigma_P(x) ds = 0 \quad {\rm for} \quad t \ge 0
\end{equation}
and 
\begin{equation}
\label{redc-}
\int_0^t \int_{\Bbb T^n} [\partial_s f(s,x)  +   \nabla_x f(s,x) \cdot (P+ \nabla_x v_{-} (P,x))] \ {\bf g}_- (s,P,x) d \sigma_P(x)  ds = 0 \quad {\rm for} \quad t \le 0
\end{equation}
\end{proposition}
\begin{proof}
We recall  $ \varphi_H^{t}|_{\mathcal{A}_P^\star}  : \mathcal{A}_P^\star \rightarrow  \mathcal{A}_P^\star$ is a one parameter group of homeomorphisms on the closed invariant graph $\mathcal{A}_P^\star$ on $\Bbb T^n$, hence
 \begin{equation}
 {\bf g}_+ \ d \sigma_P = \pi_\star  d\widetilde{m}_P(t) =  \pi_\star  (  \varphi_H^{t}  )_\star   d\widetilde{m}_P(0)    = \pi_\star  \Big(  \varphi_H^{t}|_{\mathcal{A}_P^\star}  \Big)_\star   d\widetilde{m}_P(0) =   \Big( \pi( \varphi_H^{t}|_{\mathcal{A}_P^\star} ) \Big)_\star   d\widetilde{m}_P(0)
 \end{equation}
The map $\pi (\varphi_H^{t}|_{\mathcal{A}_P^\star})  : \pi (\mathcal{A}_P^\star) \rightarrow  \pi (\mathcal{A}_P^\star)$  is a one parameter group of homeomorphisms associated with the vector field
 \begin{equation}
{\rm b}_\pm (x) := \frac{d}{dt}  \pi ( \varphi_H^{t} (x,P+ \nabla_x v_{\pm} (P,x)) ) \Big|_{t=0} = \nabla_\eta H(x, P+ \nabla_x v_{\pm} (P,x)) 
 \end{equation}
defined for any $x \in  \pi (\mathcal{A}_P^\star)$ but also in the bigger sets ${\rm dom}(\nabla_x v_{\pm} (P,\cdot))$ defined a.e. $x \in \Bbb T^n$. Here $H (x,\eta)= \frac{1}{2} |\eta|^2 + V(x)$ and thus $\nabla_\eta H (x,\eta)= \eta$. About the regularity, we have  ${\rm b}_\pm \in L^\infty (\Bbb T^n;\Bbb R^n)$. 
Write down the ODE 
\begin{equation}
\label{ODE}
\dot{\gamma} = {\rm b_\pm }(\gamma)
 \end{equation}
with $\gamma(0)=x \in {\rm dom}(\nabla_x v_{\pm} (P,\cdot))$ but remind the inclusions (see Section \ref{sec-Au})
\begin{equation}
\varphi_H^t \Big( {\rm Graph}(P + \nabla_x v_+ (P,\cdot))  \Big) \subseteq {\rm Graph}(P + \nabla_x v_+ (P,\cdot)) \quad \forall t \ge 0
\end{equation}
\begin{equation}
\varphi_H^t \Big( {\rm Graph}(P + \nabla_x v_- (P,\cdot))  \Big) \subseteq {\rm Graph}(P + \nabla_x v_- (P,\cdot)) \quad \forall t \le 0.
\label{neg-g}
\end{equation}
Thus, even if  we have the low regularity   ${\rm b}_\pm \in L^\infty (\Bbb T^n;\Bbb R^n)$ and not (in general) in the larger $W^{1,\infty} (\Bbb T^n; \Bbb R^n)$,  the equation (\ref{ODE}) is well posed and solved for $t\ge 0$ and $\gamma(0)=x \in {\rm dom}(\nabla_x v_{+} (P,\cdot))$, or in the case $t\le 0$ and $\gamma(0)=x \in {\rm dom}(\nabla_x v_{-} (P,\cdot))$.  We are now in the position to apply the same proof of Proposition 2.1 in \cite{A} and get the statement.\\
About the explicit representation of the density $g_+$  for $t \ge 0$, 
\begin{equation}
\int_{\Bbb T^n} \phi(x) {\bf g}_+ (t,P,x) d\sigma_P (x) =  \int_{\Bbb T^n} \phi(\pi( \varphi_H^{t}|_{\mathcal{A}_P^\star} )(x)) g_+ (P,x) d\sigma_P (x) =  \int_{\Bbb T^n} \phi(x) g_+ (P,\pi( \varphi_H^{-t}|_{\mathcal{A}_P^\star} )(x)) d\sigma_P (x) 
\nonumber\\
\end{equation}
since $ d\sigma_P$ is invariant under $\pi( \varphi_H^{-t}|_{\mathcal{A}_P^\star} )$. We are now looking at the Hamiltonian flow for negative times, and we recall ${\rm supp}(d\sigma_P) \subseteq \mathcal{M}_P^\star \subseteq \mathcal{A}_P^\star \subseteq {\rm Graph}(P+\nabla_x v_\pm (P,\cdot))$, thus we can choose the solution
\begin{equation}
{\bf g}_+ (t,P, x) =  g_+ (P, \pi \circ \varphi_H^{-t} (x,P+ \nabla_x v_{-} (P,x))) \quad  {\rm for} \ t\ge 0
 \end{equation}   
as we have choosen in (\ref{g+1}). The same arguments for negative times lead to
 \begin{equation}
{\bf g}_{-} (t,P, x) =  g_- (P, \pi \circ \varphi_H^{-t} (x,P+ \nabla_x v_{+} (P,x))) \quad  {\rm for} \ t\le 0.
 \end{equation}
as we have choosen  in (\ref{g-1}).
\end{proof}

\begin{remark}
Let $\psi_\hbar^\pm (s,x) := e^{-\frac{i}{\hbar}  \mathrm{Op}^w_{\hbar} (H)  s} \varphi_\hbar^\pm (x)$, define the position density  $ \rho_\hbar^\pm (s,x) := | \psi_\hbar^\pm (s,x) |^2$  and the (formal) current  density $J_\hbar^\pm (s,x) :=  \hbar \, {\rm Im} (  (\psi_\hbar^\pm)^\star  \nabla_x \psi_\hbar^\pm (s,x) )$.
The (formal) conservation law reads
\begin{equation}
\partial_t   \rho_\hbar^\pm (t,x) +  {\rm div}_x J_\hbar^\pm (t,x)= 0.
\end{equation}
In the next result we exhibit the well-posed setting.
\end{remark}

\begin{proposition}
Let $\psi_\hbar^\pm (s,x) := e^{-\frac{i}{\hbar}  \mathrm{Op}^w_{\hbar} (H)  s} \varphi_\hbar^\pm (x)$,  $\rho_\hbar^\pm (s,x) := | \psi_\hbar^\pm (s,x) |^2$. Let  $\varphi_{\hbar,\varepsilon}^\pm \in C^\infty (\Bbb T^n;\Bbb C)$ such that $\| \varphi_{\hbar,\varepsilon}^\pm - \varphi_{\hbar}^\pm  \|_{H^1} \rightarrow 0$ as $\varepsilon \rightarrow 0^+$. Define $J_{\hbar,\varepsilon}^\pm (s,x) :=  \hbar \, {\rm Im} (  (\psi_{\hbar,\varepsilon}^\pm)^\star  \nabla_x \psi_{\hbar,\varepsilon}^\pm (s,x) )$ and take a distributional limit $J_{\hbar}^\pm := \lim_{\varepsilon \rightarrow 0^+} J_{\hbar,\varepsilon}^\pm$ in $\mathcal{D}^\prime ([0,T] \times \Bbb T^n)$.
Then,
\begin{equation}
\label{cont-w}
\int_0^t \int_{\Bbb T^n} \partial_s  f(s,x)  \rho_\hbar^\pm (s,x) +   \nabla_x f(s,x) \cdot J_\hbar^\pm (s,x) \ dxds = 0 \quad \forall f \in C^\infty ([0,t] \times \Bbb T^n;\Bbb R).    
\end{equation}
\end{proposition}
\begin{proof}
This equation well posed. Indeed,
\begin{eqnarray}
E [\psi_{\hbar,\varepsilon}^\pm (s) ]  &:=&   \int_{\Bbb T^n} \frac{\hbar^2}{2} |\nabla_x \psi_{\hbar,\varepsilon}^\pm (s,x) |^2 + V(x) | \psi_{\hbar,\varepsilon} (s,x) |^2 dx 
\\
&=& \int_{\Bbb T^n} \frac{\hbar^2}{2} |\nabla_x \psi_{\hbar,\varepsilon}^\pm (0,x) |^2 + V(x) | \psi_{\hbar,0} (s,x) |^2 dx 
\\
&\rightarrow& \int_{\Bbb T^n} \frac{\hbar^2}{2} |\nabla_x \varphi_\hbar^\pm (x) |^2 + V(x) |\varphi_\hbar^\pm (x) |^2 dx  \quad {\rm as} \quad \varepsilon \rightarrow 0^+
\\ 
&=&   \int_{\Bbb T^n} \Big( \frac{1}{2} |P + \nabla_x v_\pm (P ,x)|^2 + V(x) \Big) |a_{\hbar,P}^\pm(x)|^2 dx 
+ \int_{\Bbb T^n}  \frac{\hbar^2}{2}   |\nabla_x a_{\hbar,P}^\pm(x)|^2 dx
\nonumber\\ 
&=&  \bar{H}(P) + \int_{\Bbb T^n}  \frac{\hbar^2}{2}   |\nabla_x a_{\hbar,P}^\pm(x)|^2 dx < +\infty \quad \forall \ 0 < \hbar < 1
\end{eqnarray}
since  $\hbar \| \nabla_x a_{\hbar,P}^\pm \|_{L^2} \rightarrow 0$  thanks to the setting of $a_{\hbar,P}^\pm$.\\ 
Hence $\| J_{\hbar,\varepsilon}^\pm (s,\cdot) \|_{L^1} \le  \| \psi_{\hbar,\varepsilon} (s,\cdot)  \|_{L^2} \| \hbar  \nabla_x \psi_{\hbar,\varepsilon} (s,\cdot)\|_{L^2} \le c \, \| \hbar  \nabla_x \psi_{\hbar,\varepsilon}^\pm (s,\cdot)\|_{L^2} < + \infty$ uniformly in $(\varepsilon, s) \in (0,1] \times [0,t]$.  We can take a distributional limit  $J_{\hbar}^\pm := \lim_{\varepsilon \rightarrow 0^+} J_{\hbar,\varepsilon}^\pm$ in $\mathcal{D}^\prime ([0,T] \times \Bbb T^n)$ and this gives
$$
\lim_{\varepsilon \rightarrow 0^+}  \int_{\Bbb T^n}   \nabla_x f(s,x) \cdot J_{\hbar,\varepsilon}^\pm (s,x) \ dx =  \int_{\Bbb T^n}   \nabla_x f(s,x) \cdot J_\hbar^\pm (s,x) \ dx \quad \forall s \in [0,t] 
$$
Since $\rho_{\hbar,\varepsilon}^\pm$ is weak-$\star$ converging to the unique $\rho_{\hbar}^\pm \in L^1 ([0,T] \times \Bbb T^n;\Bbb R^+)$, we deduce that  the equation (\ref{cont-w}) solved by $(\rho_{\hbar,\varepsilon}^\pm (s,x), J_{\hbar,\varepsilon}^\pm(s,x))$ (in the distributional and in the strong sense) is also fulfilled by $(\rho_{\hbar}^\pm (s,x), J_{\hbar}^\pm (s,x))$ in the distributional sense. 
\end{proof}

The last result of the section reads
\begin{corollary}
Fix $P \in \Bbb R^n$, suppose that  $v_{+} (P,\cdot) = v_{-} (P,\cdot) \in C^{2}(\Bbb T^n;\Bbb R)$ and $g(P,\cdot) \in W^{1,\infty} (\Bbb T^n; \Bbb R^+)$.  Then,  ${\bf g}_\pm$ as in Theorem \ref{th51} fulfill ${\bf g}_+ = {\bf g}_-$, 
${\bf g}_\pm \in L^1 ([0,T]; W^{1,\infty}  (\Bbb T^n; \Bbb R^+) )$  and solves the transport equation
\begin{equation}
\label{tr11}
\partial_t   {\bf g}_\pm (t,P,x)  +  (P+ \nabla_x v_{\pm} (P,x)) \cdot \nabla_x {\bf g}_\pm (t,P,x) = 0 \quad {\rm for} \quad t \in \Bbb R
\end{equation}
with inital datum ${\bf g}_\pm (0,P,x) := g(P,x)$.  
\end{corollary}
\begin{proof}
The regularity $v_{\pm} (P,\cdot) \in C^{2}(\Bbb T^n;\Bbb R)$ implies the $C^1$-regularity of the vector field $P+ \nabla_x v_{\pm} (P,\cdot)$ on $\Bbb T^n$. By standard transport PDE arguments (see for example \cite{A}) we get the above equations.
\end{proof}

\section{Appendix}      \markboth{Appendix}{}
\begin{lemma}
\label{equi-op}
Let $\hat{H}_\hbar := - \frac{1}{2} \hbar^2 \Delta_x + V(x)$, $H:=\frac{1}{2} |\eta|^2 + V(x)$  and ${\rm Op}^w_\hbar (H)$ as in  (\ref{weyl}). Then, 
\begin{equation}
\label{equi-H}
  {\rm Op}^w_\hbar (H) \psi = \hat{H}_\hbar   \psi, \quad \forall \, \psi \in C^\infty (\mathbb{T}^n;\Bbb C).
\end{equation} 
\end{lemma}
\begin{proof}
To begin,  we recall that 
\begin{equation}
\label{eq-O2}
\mathrm{Op}^w_{\hbar} (b)  \psi (x) = (b(X,\frac{\hbar}{2}D) \circ T_x \, \psi )(x).
\end{equation}
where $(T_x \psi) (y) := \psi (2y-x)$, see Section \ref{Weyl}. Morever, it is easily proved that when $H=\frac{1}{2} |\eta|^2 + V(x)$ 
\begin{equation}
\label{eq-O2}
H \Big(X,\frac{\hbar}{2}D\Big) \psi = \hat{H}_\hbar \psi
\end{equation}
for $\psi \in C^\infty (\mathbb{T}^n; \Bbb C)$ and that $(\hat{H}_\hbar   T_x \psi) (x) = \hat{H}_\hbar  \psi (x)$. Thus, by (\ref{eq-O}) and (\ref{eq-O2}) we get the statement.
\end{proof}

\begin{remark}
\label{L-reg}
The operator  $\hat{H}_\hbar  :  H^{2}  (\mathbb{T}^n; \Bbb C) \rightarrow L^2 (\mathbb{T}^n; \Bbb C)$ is linear, selfadjoint and continuous. Hence, by standard results of evolution equations in Banach spaces, the solution of the Schr\"odinger equation (\ref{eqSch}) fulfills $\psi_\hbar \in C^0 (\Bbb R; H^{2} (\Bbb T^n;\Bbb C)) \cap C^1(\Bbb R; L^2(\Bbb T^n;\Bbb C))$. The one parameter group of unitary operators $e^{-\frac{i}{\hbar}  \mathrm{Op}^w_{\hbar} (H)  t}$ can be defined on $L^{2}  (\mathbb{T}^n; \Bbb C)$ and $e^{-\frac{i}{\hbar}  \mathrm{Op}^w_{\hbar} (H)  t} \varphi \in  C^0 (\Bbb R; L^{2} (\Bbb T^n;\Bbb C))$  (see for example \cite{R-S}).
\end{remark}

\begin{remark}
We recall that  ${\bf b} \in S^m (\mathbb{R}^n \times \mathbb{R}^n)$ consist of ${\bf b} \in C^\infty (\mathbb{R}^n \times \mathbb{R}^n;\Bbb R)$ satisfying (\ref{symb00}). 
The Weyl quantization on $\Bbb R^n$ of these simbols  reads
\begin{equation}
\label{W-Rn}
{\bf b}^w (X,\hbar D) \psi (x) := (2\pi\hbar)^{-n} \int_{\mathbb{R}^{n}}  \int_{\mathbb{R}^{n}}    e^{\frac{i}{\hbar} \langle x-y,\eta\rangle}  b \Big(\frac{x+y}{2},\eta\Big)\psi(y)dyd\eta, \quad  \psi  \in \mathcal{S}  (\mathbb{R}^n;\Bbb C).
\end{equation}
Notice that in the case of $H(x,\eta) = \frac{1}{2} |\eta|^2 + V(x)$ with $V \in C^\infty (\Bbb R^n;\Bbb R)$,  it holds the equivalence $H^w (X,\hbar D) = \hat{H}_\hbar$ on the domain $\mathcal{S}  (\mathbb{R}^n;\Bbb C)$ (see for example Section 2.7 in \cite{Mar}). By the  identification $\Bbb T^n \equiv (\Bbb R / 2\pi \Bbb Z)^n$, if $V$ is $2 \pi \Bbb Z^n$-periodic then we could restrict $H^w (X,\hbar D) : C^\infty (\mathbb{T}^n; \Bbb C) \rightarrow C^\infty (\mathbb{T}^n; \Bbb C)$. Obviously,  this restriction cannot be done for all simbols in  $S^m (\mathbb{R}^n \times \mathbb{R}^n)$ which are $2 \pi \Bbb Z^n$-periodic in $x$-variables. For a more detailed and general discussion about the link between Pseudodifferential Operators on $\Bbb T^n$ and Pseudodifferential Operators on $\Bbb R^n$ which are $2 \pi \Bbb Z^n$-periodic in $x$-variables, we address the reader to Section 6 in  \cite{R-T}.
\end{remark}

The following result is shown in \cite{Yang}. 
\begin{lemma}
\label{LemmaYang}
Let $X$ be a metric space. Let $d\mu_j$ $j\in\mathbb{N}$ 
and $d\mu$ Borel probability measures on $X$ such that $d\mu_j \stackrel{\mathrm{w-}\star}{\longrightarrow}d\mu$ as $j\to+\infty.$ Let $f_k,f \colon X\longrightarrow\mathbb{R}$ ($k\in\mathbb{N}$) 
be Borel measurable functions such that
\begin{equation}
\lim_{\lambda \to+\infty} \ \sup_{k\in\mathbb{N}}  \ \int_{\{x\in X;\,\,|f_k(x)|>\lambda \}}|f_k(x)| \ d\mu_k (x)=0.
\label{eqYang1}\end{equation}
Let
\begin{equation}
E:=\Big\{  x\in X;\,\,\exists\{x_k\}_{k\in\mathbb{N}}\subset X,\,\,\lim_{k\to+\infty}x_k=x,\,\,\,\lim_{k\to+\infty}f_k(x_k)\not=f(x) \Big\}.
\label{eqYang2}\end{equation}
If $\mu(E)=0$ then
$$
\lim_{j\to+\infty}\int_Xf_j(x)d\mu_j (x) =\int_Xf(x)d\mu(x).
$$
\end{lemma}

\end{document}